\documentclass[leqno,11pt]{amsart}

\usepackage[utf8]{inputenc}
\usepackage[enumitem,theorem,hyperref]{paper_diening}

\usepackage[textsize=tiny,textwidth=3.2cm,colorinlistoftodos]{todonotes}

\usepackage[style=numeric,maxbibnames=5,backend=biber, doi=false,isbn=false,url=false,eprint=false,firstinits=true ]{biblatex}
\allowdisplaybreaks

\newcommand{\T}{{T_t}}

\newcommand{\rn}{{\mathbb R^n}}
\newcommand{\rN}{{\mathbb R^N}}
\newcommand{\rNn}{{\mathbb R^{N\times n}}}

\newcommand{\bu}{{\bfu}}
\newcommand{\ep}{\varepsilon}
\newcommand{\hh}{\mathcal H^{n-1}}

\makeatletter
\long\def\unmarkedfootnote#1{{\long\def\@makefntext##1{##1}\footnotetext{#1}}}
\makeatother

\numberwithin{equation}{section}
\numberwithin{theorem}{section}

\allowdisplaybreaks

\usepackage{amsfonts}
\usepackage{amsxtra}

\title[A pointwise differential inequality    and    nonlinear elliptic systems]
{A pointwise differential inequality    and  second-order regularity for nonlinear elliptic systems} 
\numberwithin{equation}{section}
\author[Kh.Balci]{Anna Kh.Balci}
\address{Fakult\"at für Mathematik, University Bielefeld,
Universit\"atsstrasse 25, 33615 Bielefeld, Germany}
\email{akhripun@math.uni-bielefeld.de }
\author[Cianchi]{ Andrea Cianchi}
\address{Dipartimento di Matematica e Informatica \lq\lq U. Dini", Universit\`a di Firenze,
Viale Morgagni 67/A, 50134 Firenze, Italy}
\email{andrea.cianchi@unifi.it}
\author[Diening]{ Lars Diening}
\address{Fakult\"at für Mathematik, University Bielefeld,
Universit\"atsstrasse 25, 33615 Bielefeld, Germany}
\email{lars.diening@uni-bielefeld.de}
\author[Maz'ya]{Vladimir Maz'ya}
\address{Department of Mathematics, Link\"oping University,E-581 83 Link\"oping, Sweden
and  Peoples Friendship University of Russia (RUDN University),
6 Miklukho-Maklay St, Moscow, 117198, Russian Federation}
\email{vladimir.mazya@liu.se}

\date{}

\begin{document}

\begin{abstract}
  A sharp pointwise differential inequality for vectorial   second-order partial differential operators, with Uhlenbeck structure, is offered. As a consequence, optimal second-order regularity properties of solutions to nonlinear elliptic systems in domains in $\rn$ are derived. Both local and global estimates are established. Minimal   assumptions on the boundary of the domain are required for the latter. In the  special case of the $p$-Laplace system, our conclusions broaden the range of the  admissible values of the exponent $p$ previously known.
\end{abstract}

\maketitle

\unmarkedfootnote {
}


\section{Introduction}\label{intro}

A classical identity, which  links the Laplacian $\Delta \bfu$ of a   vector-valued function
$\bfu \in C^3(\Omega, \rN)$  to its Hessian $\nabla^2 \bfu$, tells us that
\begin{equation}\label{linearvector}
  |\Delta \bfu|^2  =  {\rm {div}} \Big(
  (\Delta \bfu)^T \nabla \bfu
  - \tfrac 12   \nabla |\nabla \bfu|^2\Big)
  +   |\nabla ^2 \bfu|^2 \quad \hbox{in $\Omega$,}
\end{equation}
where $\Omega$ is an open set in $\rn$. 
Here, and in what follows, $n\geq 2$, $N \geq 1$, and the gradient $\nabla \bfu$ of a function $\bfu : \Omega \to \setR^N$ is regarded as the matrix in   $\setR^{N\times n}$ whose rows are the gradients in $\setR^{1 \times n}$ of the components $u^1, \, \dots \, , u^N$ of $\bfu$. Moreover, the suffix \lq\lq $T$" stands for transpose. 
%
\par Identity  \eqref{linearvector} can be found as early as more than one century ago in \cite{Bernstein}  for $n=2$ -- see also \cite{Sobolev, Grisvard}.  It has applications, for instance, in the second-order $L^2$-regularity theory for solutions to the Poisson system for the  Laplace operator
\begin{equation}\label{poisson}
  - \Delta \bfu = \bff \qquad \text{in $\Omega$.}  
\end{equation}
Indeed,  identity \eqref{linearvector} enables one to bound the integral of $|\nabla ^2 \bfu|^2$ over some set in $\Omega$ by the integral of $|\Delta \bfu|^2$ over the same set, plus a boundary integral involving the expression under the divergence operator. 
Of course, since the equations in the linear system \eqref{poisson} are uncoupled, its   theory is reduced to that of its single equations.
\par The second-order regularity theory of nonlinear equations and systems is much less developed, yet for the basic $p$-Laplace equation or system
\begin{equation}\label{plapl}
  - {\bf div} (|\nabla \bfu|^{p-2}\nabla \bfu ) = \bff \qquad \text{in $\Omega$,}
\end{equation}
where $p>1$ and $\lq\lq {\bf div} "$ denotes the $\setR^N$-valued divergence operator. Standard results concern weak differentiability properties  of the expression $|\nabla \bfu|^{\frac{p-2}2}\nabla \bfu$. They   trace back to \cite{Uhl} for $p>2$, and to \cite{AcFu, CDiB} for every $p >1$. The case of a single equation was earlier considered in \cite{Ural}. Further developments are in   \cite{BC, Cel, CGM}. 

As demonstrated in several more recent contributions,
the regularity of solutions to $p$-Laplacian type equations and
systems is often most neatly described in terms of the expression
$|\nabla \bfu|^{{p-2}}\nabla \bfu$ appearing under the divergence
operator in \eqref{plapl}. This surfaces, for instance, from BMO and H\"older bounds of \cite{DKS}, potential estimates of \cite{KM1},  rearrangement inequalities of \cite{cm2},  pointwise oscillation estimates of \cite{BCDKS18}, regularity results up to the boundary of \cite{BCDS19global}. Further results in this connection can be found e.g. \cite{AKM, 
  cmARMA, KuuMin18}.


Differentiability properties of $|\nabla \bfu|^{{p-2}}\nabla \bfu$ have customarily been
detected under strong regularity assumptions on the right-hand side
$\bff$. 
This is the case of \cite{Lou}, where local solutions are
considered. High regularity of the right-hand side is also assumed
\cite{Dasc}, where results for boundary value problems can be found
under smoothness assumptions on $\partial \Omega$.  Both papers
\cite{Lou} and \cite{Dasc} deal with scalar problems, i.e. with the
case when $N=1$.  Fractional-order regularity of the gradient of
solutions to quasilinear equations of $p$-Laplacian type has been
studied in \cite{SimonJ}, and in the more recent contributions
\cite{AKM, BSY, Cel1, Mi1, Mis}. The question of fractional-order
regularity of the quantity $|\nabla \bfu|^{{p-2}}\nabla \bfu$, when
$N=1$ and the right-hand side of equation \eqref{plapl} is in
divergence form, is addressed in \cite{BaDieWei}, where, in
particular, sharp results are obtained for $n=2$.

\par
Optimal second-order $L^2$-estimates for solutions to a class of problems, including \eqref{plapl} for every $p >1$,  in the scalar case, have recently been established in   \cite{CiMa_ARMA}. 
Loosely speaking, these estimates tell us that 
$|\nabla \bfu|^{p-2}\nabla \bfu  \in W^{1,2}$ if and only if $\bff \in L^2$. Such a  property is shown to hold both locally, and, under minimal regularity assumptions on the boundary, also globally. Parallel results are derived  in \cite{CiMa_JMPA} for vectorial problems, namely for $N\geq 2$, but for the restricted range of powers $p >\frac 32$. The results of    \cite{CiMa_JMPA} and \cite{CiMa_ARMA} rely upon the idea that, in the nonlinear case, the role of the pointwise identity \eqref{linearvector} can be performed by   a pointwise inequality. The latter amounts to a bound from below for the square of the right-hand side of \eqref{plapl} by  the square of the derivatives of $|\nabla \bfu|^{p-2}\nabla \bfu$, plus an expression in divergence form. The restriction for the admissible values of $p$ in the vectorial case stems from this pointwise inequality. 
\par
In the present paper we offer an enhanced pointwise inequality in the same spirit, with best possible constant, for a class of nonlinear differential operators of the form $- {\bf div}(a(|\nabla \bfu |) \nabla \bfu)$. The relevant inequality holds under general assumptions on the function $a$, which also allow  growths that are not necessarily of power type. Importantly,  our inequality improves the available  results even in the case when the operator is the $p$-Laplacian, namely when $a(t)=t^{p-2}$. In particular, for this special choice, it entails the  existence of a  constant $c>0$ such that
\begin{align}
  \label{eq:aux1}
\big|{\rm {\bf div}} (|\nabla \bfu|^{p-2}\nabla \bfu )\big|^2 \geq 
{\rm {div}} \Big[|\nabla \bfu|^{2(p-2)}
\Big((\Delta \bfu)^T \nabla \bfu 
- \tfrac 12    \nabla |\nabla \bfu|^2\Big)\Big]  
+ c\,  |\nabla \bfu|^{2(p-2)} |\nabla ^2 \bfu|^2 
%
\end{align}
in $\{\nabla \bfu \neq 0\}$ if and only either $N=1$ and $p>1$, or $N \geq 2$ and  $p>2(2-\sqrt{2})\approx 1.1715$. 

The differential inequality to be presented, in  its general version, is   the crucial point of departure in our proof of the  local and global $W^{1,2}$-regularity for the expression $a(|\nabla \bfu |) \nabla \bfu$ for systems of the form
\begin{equation} \label{system-a}
- {\rm {\bf  div}}( a(|\nabla \bu|) \nabla {\bf u} ) = {\bf f}  \quad {\rm in}\,\,\, \Omega.
\end{equation}
Regularity issues for equations and systems driven by non standard  nonlinearities, encompassing \eqref{system-a}, are nowadays the subject of a rich literature.  A non exhaustive sample of contributions along this direction of research includes
\cite{ACCZ,  BalDieGioPas20,Baroni, BeMi, BMSV, CKP, Chl, Ci_CPDE, Ci_AIHP, DiMa, DKS, DSV, GSZ, HHT,  Lieber, Mar, Ta}.

Let us incidentally note that system \eqref{system-a} is the Euler equation  of   the   functional
\begin{equation}\label{functional}
J(\bu) = \int _\Omega B(|\nabla \bu|) - {\bf f}\cdot \bu \,\,
dx.
\end{equation}
Here,  the dot $\lq\lq \, \cdot \, "$ stands for  scalar product, and 
$B:[0, \infty) \to [0, \infty)$ is the function defined as
\begin{equation}\label{B}
B(t) = \int_0^t b(s)\, ds \qquad \text{for $t \geq 0$,}
\end{equation}
where the function $b: [0, \infty) \to [0, \infty)$ is given by 
\begin{equation}\label{b}
b(t)= a(t) t \qquad \hbox{for $t >0$,}
\end{equation}
and $b(0)=0$. 
\\
Under the assumptions to be imposed on $a$, the function $B$ and  the functional $J$  turn out to be strictly  convex.  In particular, if $a(t)=t^{p-2}$, then $B(t)=\frac 1p t^p$, and $J$ agrees with the usual energy functional associated with the $p$-Laplace system \eqref{plapl}.
 
We shall focus on  the case when $N \geq 2$, the case of
equations being  already fully covered by the results of \cite{CiMa_ARMA}. 
In particular, our regularity results apply  to the $p$-Laplacian system \eqref{plapl} for every 
\begin{equation}\label{sharpp}
p>2(2-\sqrt{2}) \approx 1.1715.
\end{equation}
Hence, we extend the range of the admissible exponents $p$ known until now, which was  limited to $p>\frac 32$. 

  In the light of the pointwise inequality~\eqref{eq:aux1}, the lower bound \eqref{sharpp}  for $p$  is 
optimal for our approach   to the second-order regularity of solutions to the $p$-Laplace system \eqref{plapl}. The question of whether such a restriction is really indispensable for this regularity, or it can be dropped  as in the case when $N=1$, where    every $p>1$ is admitted, is   an open challenging problem.

\section{Main results}\label{S:main}

%
The statement of the general differential inequality requires a few notations.
Given a positive function
 $a \in C^1((0, \infty))$,  
we define the  indices  
 \begin{equation}\label{ia}
i_a= \inf _{t >0} \frac{t a'(t)}{a(t)} \qquad \hbox{and} \qquad
s_a= \sup _{t >0} \frac{t a'(t)}{a(t)},
\end{equation}
where $a'$ stands for the derivative of $a$. Plainly, if $a(t)=t^{p-2}$, then $i_a=s_a=p-2$.

Moreover, we denote, for $N\geq 1$ the continuously increasing function 
 $\kappa_N : [1, \infty) \to \mathbb R$ as
\begin{equation}\label{kappa1}
\kappa_1 (p) = 
\begin{cases}
                 (p-1)^2 &\qquad \text{if $p \in [1,2)$}
                 \\
                1
 &\qquad  \text{if $p \in [2, \infty),$}
               \end{cases}
%
%
\end{equation}
if $N=1$, and
\begin{equation}\label{cp}
\kappa_N (p) = \begin{cases}
                 1- \frac{1}8(4-p)^2  &\qquad  \text{if
                   $p \in [1, \frac 43)$}
                 \\
                 (p-1)^2 &\qquad \text{if $p \in [\frac 43,2)$}
                 \\
                1
 &\qquad  \text{if $p \in [2, \infty),$}
               \end{cases}
\end{equation}
 if $N \geq 2$.


\begin{theorem}\label{lemma1}{\rm {\bf [General pointwise inequality]}}
Let $n \geq 2$ and  $N \geq 1$.  Let $\Omega$ be an open set in $\rn$ and let $\bfu \in C^3(\Omega, \rN )$.  Assume that the function  $a \in C^0([0, \infty))$ is   such that: 
\begin{equation}\label{positive} 
a(t)>0 \qquad \text{ for $t>0$,}  
\end{equation}
\begin{equation}\label{fundhp}
i_a\geq -1,
\end{equation}
and 
\begin{equation}\label{bC1}
b \in C^1([0, \infty)),
\end{equation}
where $b$ is the function defined by \eqref{b}.
Then  
\begin{align}\label{pointwise}
\big|{\rm {\bf div}} \big(a(|\nabla \bfu|)\nabla \bfu \big)\big|^2  \geq {\rm {div}} \Big[a(|\nabla \bfu|)^2  \Big(
 (\Delta \bfu)^T \nabla \bfu
- \tfrac 12   \nabla |\nabla \bfu|^2\Big)\Big]
 +  \kappa_N (i_a+2)  a(|\nabla \bfu|)^2 |\nabla ^2 \bfu|^2  
%
\end{align}
in $\Omega$,  where $\kappa_N $ is defined as in \eqref{kappa1}-\eqref{cp}.
Moreover, the constant $ \kappa_N (i_a+2)$ is sharp.
\\ If $a$   is   just defined in $(0,\infty)$, $a \in C^1((0, \infty))$, and  conditions \eqref{positive} and \eqref{fundhp} are fulfilled,
then inequality \eqref{pointwise} continues to hold  in the set $\{\nabla \bfu \neq 0\}$.
%
\end{theorem}

\begin{remark}\label{ahat}
{\rm Observe that 
 the assumption \eqref{bC1}
need not be fulfilled by the functions $a$ appearing in the elliptic  systems  \eqref{system-a} to be considered. Such an assumpton fails, for instance, when $a(t)=t^{p-2}$ with $1<p<2$. This calls for a regularization argument for $a$ in our applications of inequality \eqref{pointwise} to the solutions to the systems in question. The solutions to the regularized systems will also enjoy the  smoothness properties required   on the function $\bfu$ in Theorem \ref{lemma1}. On the other hand, the functions $a$ in the original systems satisfy the conditions required in the last part of the statement of Theorem \ref{lemma1} for the validity of inequality \eqref{pointwise} outside the set $\set{\nabla \bfu = 0}$ of critical points of the function $\bfu$. 
}
\end{remark}

Specializing  Theorem \ref{lemma1} to the case in which $a(t)=t^{p-2}$   yields the following inequality  for the $p$-Laplace operator  we alluded to in Section \ref{intro}.

\begin{corollary}\label{pointiwisep}{\rm {\bf [Pointwise inequality for the $p$-Laplacian]}}
Let $n \geq 2$ and $N \geq 1$. Let $\Omega$ be an open set in $\rn$ and let $\bfu \in C^3(\Omega, \rN )$.  Assume that $p\geq1$.
Then  
\begin{align}\label{pointwisep}
\big|{\rm {\bf div}} (|\nabla \bfu|^{p-2}\nabla \bfu )\big|^2  \geq 
{\rm {div}} \Big[|\nabla \bfu|^{2(p-2)}
\Big((\Delta \bfu)^T \nabla \bfu
- \tfrac 12    \nabla |\nabla \bfu|^2\Big)\Big]
 + \kappa_N(p)  |\nabla \bfu|^{2(p-2)} |\nabla ^2 \bfu|^2  
%
\end{align}
 in $\{\nabla \bfu \neq 0\}$.
Moreover, the constant $ \kappa_N(p)$ is sharp.
\end{corollary}

Notice that,
if $N=1$, then 
 \begin{equation}\label{>0bis}
\kappa_1(p)>0 \quad \text{if $p >1$},
\end{equation}
whereas, if $N \geq 2$,
\begin{equation}\label{>0}
\kappa_N (p) >0 \quad \text{ if }  \quad p>2(2-\sqrt{2}).
\end{equation}
The gap
between \eqref{>0bis} and \eqref{>0} is responsible for the different implications of  inequality \eqref{pointwise}  in view of second-order $L^2$-estimates  for solutions to 
\begin{equation}\label{localeq}
- {\rm {\bf  div}}( a(|\nabla \bu|) \nabla {\bf u} ) = {\bf f}  \quad {\rm in}\,\,\, \Omega\,,
\end{equation}
according to whether  $N=1$ or $N \geq 2$. Indeed,  inequality \eqref{pointwise} is of use for this purpose only if $\kappa_N(i_a+2)>0$.
%

%
%

Since we are concerned with $L^2$-estimates, the datum  ${\bf f}$ in \eqref{localeq} is assumed to be merely square integrable. Solutions in a suitably generalized sense have thus to be considered. For instance,   the existence of standard weak solutions to the $p$-Laplace system \eqref{plapl} is only guaranteed if $p\geq \frac{2n}{n+2}$. In the scalar case, various definitions of solutions -- entropy solutions, renormalized solutions, SOLA -- that allow for right-hand sides that are just integrable functions, or even finite measures, are available in the literature, and turn out to be a posteriori equivalent. Note that these solutions need not be even weakly differentiable. The case of systems is more delicate and has been less investigated. A notion of solution, which is well tailored for our purposes and will be adopted, is patterned on the approach of \cite{DHM}. Loosely speaking, the solutions in question are only approximately differentiable, and are pointwise limits of solutions to approximating problems with smooth right-hand sides. 

The outline of the derivation of the second-order $L^2$-bounds for these solutions to system \eqref{localeq} via 
Theorem \ref{lemma1} is analogous to the one of \cite{CiMa_ARMA}. However, new technical obstacles have to be faced, due to the non-polynomial growth of the coefficient $a$ in the differential operator. In particular, an $L^1$-estimate, of independent interest, for the expression  $a(|\nabla \bu|) \nabla {\bf u}$ for merely integrable data $\bff$ is established. Such an estimate is already available in the literature for equations, but seems to be new for systems,  and its proof requires an ad hoc Sobolev type inequality in Orlicz spaces.

\smallskip
\par
Our local estimate for  system \eqref{localeq} reads as follows. In the statement, $B_R$ and $B_{2R}$ denote concentric balls, with radius $R$ and $2R$, respectively. 

\begin{theorem}\label{secondloc} {\rm {\bf [Local estimates]}}
Let $\Omega$ be an open set in $\mathbb R^n$, with $n \geq 2$, and let $N \geq 2$. 
Assume that the function $a: (0, \infty) \to (0, \infty)$ is continuously differentiable, and satisfies
\begin{equation}\label{inf}
i_a>2(1-\sqrt 2)\,,
\end{equation}
and  
\begin{equation}\label{sup}
s_a< \infty\,.
\end{equation}
Let   $\bff \in L^2_{\rm loc}(\Omega, \rN)$ and
   let $\bu$ be an approximable  local solution   to  system
\eqref{localeq}.
Then
\begin{equation}\label{secondloc1} a(|\nabla \bu|) \nabla \bu
\in W^{1,2}_{\rm loc}(\Omega, \rNn ),
\end{equation}
and there exists a constant $C=C(n, N, i_a, s_a)$ such that
\begin{align}
  \label{eq:secondloc2}
 R^{-1}\bignorm{a(|\nabla \bu|) \nabla \bu}_{L^2(B_R, \rNn)} & + \,\bignorm{\nabla \big(a(|\nabla \bu|) \nabla \bu\big)}_{L^2(B_R, \rNn)}
 \\ \nonumber  &
   \leq C \Big(\,\|\bff\|_{L^2(B_{2R}, \rN)} + R^{-\frac n2-1}\|a(|\nabla \bu|)\nabla \bu\|_{L^{1}(B_{2R}, \rNn)}\Big).
\end{align}
for any ball $B_{2R} \subset \subset \Omega$.
\end{theorem}

\begin{remark}\label{scaling}
{\rm 
In particular, if $\Omega= \mathbb R^n$ and, for instance,  $a(|\nabla \bu|)\nabla \bu \in L^1(\mathbb R^n, \rNn)$, then passing to the limit in inequality  \eqref{eq:secondloc2} as $R\to \infty$ tells us that }
\begin{equation}\label{Rn}
    \bignorm{\nabla \big(a(|\nabla \bu|) \nabla \bu\big)}_{L^2(\mathbb R^n, \rNn)}
   \leq C  \|\bff\|_{L^1(\mathbb R^n, \rN)}.
\end{equation}
\end{remark}

We next deal with global estimates for solutions to   system \eqref{localeq}, subject to Dirichlet homogeneous boundary conditions. Namely, we consider solutions to problems of the form
\begin{equation}\label{eqdir}
\begin{cases}
-{\rm {\bf  div}} (a(|\nabla \bu|)\nabla {\bfu} ) = {\bff}  & {\rm in}\,\,\, \Omega\\
 {\bfu} =0  &
{\rm on}\,\,\,
\partial \Omega \,.
\end{cases}
\end{equation}
 
As shown by classical counterexamples, yet in the linear case, global  estimates involving second-order derivatives  of solutions can only hold under suitable  regularity assumptions on $\partial \Omega$. Specifically,   information on the (weak) curvatures of  $\partial \Omega$
 is relevant in this connection.  Convexity of the domain $\Omega$, which results in a positive semidefinite second fundamental form of $\partial \Omega$,  is well known to ensure bounds in $W^{2,2}(\Omega, \rNn)$ for the solution $\bu$ to the homogeneous  Dirichlet  problem associated with the linear system \eqref{poisson} in terms of the $L^2(\Omega, \rN)$ norm of $\bff$ -- see \cite{Grisvard}. The following result provides us with an analogue for  problem \eqref{eqdir}, for the same class of nonlinearities $a$ as in Theorem \ref{secondloc}.

\begin{theorem}\label{secondconvexteo} {\rm {\bf [Global estimates in convex domains]}}
Let $\Omega$ be any  bounded convex open set in $\rn$, with $n \geq 2$, and  let $N \geq 2$. Assume that the function $a: (0, \infty) \to (0, \infty)$ is  continuously differentiable  and fulfills conditions \eqref{inf} and \eqref{sup}.
Let ${\bff} \in
L^2(\Omega, \rN )$ and   let $\bu$ be an approximable  solution   to  the
Dirichlet problem \eqref{eqdir}.
%
%
Then
\begin{equation}\label{secondconvex} a(|\nabla \bu|) \nabla \bu
\in W^{1,2}(\Omega, \rNn ),
\end{equation}
and
\begin{equation}\label{secondconvex1} C_1  \|{\bff}\|_{L^2(\Omega, \rN)} \leq \|a(|\nabla \bu|) \nabla \bu\|_{W^{1,2}(\Omega, \rNn)} \leq C_2  \|{\bff}\|_{L^2(\Omega, \rN)} 
\end{equation}
for some positive constants $C_1=C_1(n,  N, i_a, s_a)$ and  $C_2=C_2(N, i_a, s_a,  \Omega)$.
%
%
%
\end{theorem}

The global assumption on the signature of the second fundamental form of $\partial \Omega$ entailed by the convexity of $\Omega$ can be replaced by  local   conditions on the relevant fundamental form. This is the subject of Theorem \ref{seconddir}. 
\par The finest assumption on $\partial \Omega$ that we are able to allow for amounts to a decay estimate of the integral of its weak curvatures over subsets of $\partial \Omega$ whose diameter approaches zero, in terms of their capacity. 
 Specifically, suppose that $\Omega$ is a bounded Lipschitz domain such that $\partial \Omega \in W^{2,1}$. This  means
 that the domain $\Omega$ is    locally the subgraph of a Lipschitz continuous function of $(n-1)$ variables, which is also twice weakly differentiable.
Denote by 
$\mathcal B$ the weak second fundamental form on $\partial \Omega$, by
$|\mathcal B|$ its norm, and set 
\begin{equation}\label{defK}
\mathcal K_\Omega(r) =
\sup_{
\substack{
 E\subset \partial \Omega \cap B_r(x) \\
  x\in \partial \Omega
}
}
 \frac{\int _E |\mathcal B|d\hh}{{\rm cap}_{B_1(x)} (E)}\qquad \hbox{for $r\in (0,1)$}\,.
\end{equation}
Here, $B_r(x)$ stands for the ball centered at $x$, with radius $r$,  the notation ${\rm cap}_{B_1(x)}(E)$ is adopted for the capacity of the set $E$ relative to the ball $B_1(x)$, and $\hh$ is the $(n-1)$-dimensional Hausdorff measure. The decay we hinted  to above consists in a smallness condition on the limit at as $r\to 0^+$ of the function $\mathcal K_\Omega(r)$. The smallness depends on $\Omega$ through its diameter $d_\Omega$ and its Lipschitz characteristic $L_\Omega$.
The latter quantity is defined as the maximum among the Lipschitz constants of the functions that locally describe the intersection of $\partial \Omega$ with balls centered on $\partial \Omega$, and the reciprocals of their radii. Here, and in similar occurrences in what follows, the dependence of a constant on $d_\Omega$ and   $L_\Omega$ is understood just via an upper bound for them.
\par  Theorem \ref{seconddir}   also 
 provides us with an ensuing alternate assumption on $\partial \Omega$, which  only depends on integrability properties of the weak curvatures of $\partial \Omega$. Precisely, it requires the membership of  $|\mathcal B|$   in a suitable function space $X(\partial \Omega)$ over $\partial \Omega$ defined in terms of weak type  norms, and a smallness  
  condition on the decay of these norms of  $|\mathcal B|$ over  balls  centered on $\partial \Omega$. This membership will be denoted by $\partial \Omega \in W^2X$. The relevant weak space  is defined as
\begin{equation}\label{Xspace}
X (\partial \Omega)= \begin{cases} L^{n-1, \infty}(\partial \Omega) & \quad \hbox{if $n \ge 3$,}
\\
L^{1, \infty} \log L  (\partial \Omega)& \quad \hbox{if $n =2$.}
\end{cases}
\end{equation}
Here, $L^{n-1, \infty}(\partial \Omega)$ denotes the weak-$L^{n-1}(\partial \Omega)$ space, and $L^{1, \infty} \log L  (\partial \Omega)$ denotes the weak-$L\log L   (\partial \Omega)$ space  (also called Marcinkiewicz spaces), with respect to the $(n-1)$-dimensional Hausdorff measure.

\begin{theorem}\label{seconddir} {\rm {\bf [Global estimates under  minimal  boundary regularity]}}
 Let $\Omega$ be a bounded Lipschitz  domain in $\rn$, $n \geq 2$, such that   $\partial \Omega \in  W^{2,1}$, and   let $N \geq 2$.
Assume that the function $a: (0, \infty) \to (0, \infty)$ is  continuously differentiable   and fulfills conditions \eqref{inf} and \eqref{sup}.
Let ${\bff} \in
L^2(\Omega, \rN )$ and   let $\bu$ be an approximable  solution   to  the
Dirichlet problem \eqref{eqdir}.
%
\\ (i) There exists a constant $c=c(n, N, i_a, s_a, L_\Omega, d_\Omega)$ such that, if
\begin{equation}\label{capcond}
\lim _{r\to 0^+} \mathcal K _\Omega (r) < c,
\end{equation}
then $a(|\nabla \bu|) \nabla \bu
\in W^{1,2}(\Omega, \rNn )$, and inequality \eqref{secondconvex1} holds.
\\ (ii) 
Assume, in addition, that $\partial \Omega \in W^2X$, where $X(\partial \Omega)$ is the space defined by \eqref{Xspace}. There exists a constant $c=c(n, N, i_a, s_a, L_\Omega, d_\Omega)$ such that, if   
\begin{equation}\label{smalln}
 \lim _{r\to 0^+} \Big(\sup _{x \in \partial \Omega} \|\mathcal B \|_{X(\partial \Omega \cap B_r(x))}\Big) < c \,,
\end{equation}
then $a(|\nabla \bu|) \nabla \bu
\in W^{1,2}(\Omega, \rNn )$, and inequality \eqref{secondconvex1} holds.
\end{theorem}

\begin{remark}\label{sharpcond}{\rm
We emphasize that the assumptions on $\partial \Omega$ in Theorem \ref{seconddir} are essentially sharp. For instance, the mere finiteness of the limit in \eqref{capcond} is not sufficient for the conclusion to hold. As shown in  \cite{Ma67, Ma73}, there exists a one-parameter family of domains $\Omega$ such that  $\mathcal K_{\Omega}(r)<\infty$ for $r\in (0,1)$ and the solution to the homogeneous Dirichlet problem for \eqref{poisson}, with a smooth right-hand side $\bff$,  belongs to $W^{2,2}({\Omega})$ only for those values of the parameter which make the limit  in  \eqref{capcond} smaller than a critical (explicit) value.
\\ A similar phenomenon occurs in connection with 
assumption \eqref{smalln}. An example from \cite{KrolM} applies to demonstrate its optimality yet for the scalar $p$-Laplace equation.
Actually, open sets $\Omega \subset \mathbb R^3$, with   $\partial {\Omega} \in W^2L^{2, \infty}$, are displayed where the solution $\bfu$ to the homogeneous Dirichlet problem for \eqref{plapl}, 
with $N=1$, $p\in (\tfrac 32,2]$ and a smooth right-hand side $\bff$, is such that $|\nabla u|^{p-2}\nabla u \notin W^{1,2}(\Omega)$. This lack of regularity is due to the fact that the  limit in  \eqref{smalln}, though finite, is not small enough. Similarly, if $n=2$ there exist open sets $\Omega$, with 
$\partial {\Omega} \in  W^2L^{1, \infty} \log L$, for which the limit in  \eqref{smalln}   exceeds  some threshold,
and  where the  solution to  the    homogeneous Dirichlet problem for \eqref{poisson},  with  a smooth right-hand side, does not belong to $W^{2,2}({\Omega})$ -- see  \cite{Ma67}.}
\end{remark}

\begin{remark}\label{finer}{\rm The one-parameter family of domains $\Omega$ mentioned in the first part of Remark \ref{sharpcond} with regard to condition \eqref{capcond} is such that $\partial {\Omega} \notin   W^2L^{n-1, \infty}$ if $n \geq 3$. Hence, assumption  \eqref{smalln}  is not fulfilled even for those values of the parameter which render \eqref{capcond} true. This shows that the latter assumption is indeed weaker than \eqref{smalln} .
}
\end{remark}

\begin{remark}\label{rembound}
{\rm 
Condition \eqref{smalln} certainly holds  if $n \geq 3$ and $\partial {\Omega} \in W^{2,n-1}$, and if $n=2$ and $\partial {\Omega} \in W^{2}L\log L$ (and hence,  if  $\partial {\Omega} \in W^{2,q}$ for some $q>1$).  This is due to the fact that, under these assumptions, the limit in  \eqref{smalln} vanishes. 
In particular, assumption  \eqref{smalln} is satisfied if $\partial \Omega \in C^2$.}
\end{remark}


\section{The pointwise inequality}\label{S:point}

This section is devoted to the proof of Theorem \ref{lemma1}, which is split in several lemmas. The point of departure is  a pointwise  identity, of possible independent use, stated in Lemma \ref{peq}.

Given a positive function
 $a \in C^1(0, \infty)$, we
define the function $Q_a: [0, \infty) \to \mathbb R$  
\begin{equation}\label{I}
Q_a(t) = \frac{t a'(t)}{a(t)}\qquad \text{for $t>0$.}
\end{equation}
Hence, 
 \begin{equation}\label{Qa}
i_a= \inf _{t >0} Q_a(t) \qquad \hbox{and} \qquad
s_a= \sup _{t >0} Q_a(t),
%
\end{equation}
where  $i_a$ and $s_a$ are the indices given by \eqref{ia}.

%
%

\begin{lemma}\label{peq} Let $n$, $N$, $\Omega$ and $\bfu$ be as in Theorem \ref{lemma1}.
Assume that the function  $a \in C^0([0, \infty))$  and satisfies conditions \eqref{positive}--\eqref{bC1}.
Then  
\begin{align}\label{pointwiseq}
\big|{\rm {\bf div}} (a(|\nabla \bfu|)\nabla \bfu )\big|^2  &= {\rm {div}} \Big[a(|\nabla \bfu|)^2  \Big(
(\Delta \bfu)^T \nabla \bfu
- \tfrac 12   \nabla |\nabla \bfu|^2\Big)\Big]
\\ \nonumber &  +
a(|\nabla \bfu|)^{2}  \Bigg[ |\nabla ^2 \bfu|^2 
 +2 Q_a(|\nabla \bfu|)|\nabla |\nabla \bfu||^2 +Q_a(|\nabla \bfu|)^2\bigg| \frac{\nabla \bfu}{|\nabla \bfu|} (\nabla |\nabla \bfu|)^T \bigg|^2\Bigg] \quad \text{in $\Omega$, }
\end{align}
where the last two addends in square brackets on the right-hand side of equation \eqref{pointwiseq} have to interpreted as $0$   if $\nabla \bfu =0$.
\\ If $a$   is   just defined in $(0,\infty)$, $a \in C^1((0, \infty))$, and  conditions \eqref{positive} and \eqref{fundhp} are fulfilled,
then inequality \eqref{pointwise} continues to hold  in the set $\{\nabla \bfu \neq 0\}$.
\end{lemma}

The next corollary follows from Lemma \ref{peq}. applied with
 $a(t)=t^{p-2}$.

\begin{corollary}\label{plapleq} Let $n$, $N$, $\Omega$ and $\bfu$  be as in Theorem \ref{lemma1}.
Assume that $p\geq1$.
Then  
\begin{align}\label{pointwisepeq}
\big|{\rm {\bf div}} (|\nabla \bfu|^{p-2}\nabla \bfu )\big|^2 & =
{\rm{div}} \Big[|\nabla \bfu|^{2(p-2)}
\Big((\Delta \bfu)^T \nabla \bfu
- \tfrac 12    \nabla |\nabla \bfu|^2\Big)\Big]
\\ \nonumber & +
|\nabla \bfu|^{2(p-2)}  \Bigg[ |\nabla ^2 \bfu|^2 
 +2 (p-2)|\nabla |\nabla \bfu||^2 +(p-2)^2\bigg| \frac{\nabla \bfu}{|\nabla \bfu|} (\nabla |\nabla \bfu|)^T \bigg|^2\Bigg]
\end{align}
in $\{\nabla \bfu \neq 0\}$.  
\end{corollary}

\begin{proof}[Proof of Lemma \ref{peq}] 
The following chain can be deduced via  straightforward computations:
\begin{align}\label{point1}
\big| {\rm {\bf div}}  \big(a(|\nabla \bfu|)\nabla \bfu \big)\big|^2  &  =  \big|
  a(|\nabla \bfu|) \Delta \bfu+   a'(|\nabla \bfu|) \nabla \bfu (\nabla |\nabla \bfu|)^T \big|^2 
\\ \nonumber & =  \  a(|\nabla \bfu|)^2 \big(|\Delta \bfu|^2
- |\nabla ^2 \bfu|^2\big) +  a(|\nabla \bfu|)^2   |\nabla ^2 \bfu|^2 +
 \\ \nonumber & \quad +   a'(|\nabla \bfu|)^2 | \nabla \bfu (\nabla |\nabla \bfu|)^T \big|^2 
+ 2
  a(|\nabla \bfu|)   a'(|\nabla \bfu|) \Delta \bfu \cdot \nabla \bfu (\nabla |\nabla \bfu|)^T
 \\ \nonumber & =  a(|\nabla \bfu|)^2\Big( {\rm {div}} (( \Delta \bfu)^T \nabla \bfu) -  \tfrac 12{\rm {div}} (\nabla |\nabla \bfu|^2)\Big)
%
%
%
+
  a(|\nabla \bfu|)^2   |\nabla ^2 \bfu|^2 +
 \\ \nonumber & \quad +   a'(|\nabla \bfu|)^2 \big| \nabla \bfu (\nabla |\nabla \bfu|)^T \big|^2 
+ 2
 a(|\nabla \bfu|)  a'(|\nabla \bfu|) \Delta \bfu \cdot \nabla \bfu (\nabla |\nabla \bfu|)^T.
\end{align}
Notice that equation \eqref{point1} also holds at the points where $|\nabla \bfu|=0$, 
provided  the terms  involving the factor $a'(|\nabla \bfu|)$   are intepreted as $0$.  This is due to the fact that all the terms in question also contain the factor $\nabla \bfu$ and,  
by assumption \eqref{bC1},  
$$\lim _{t\to 0^+} a'(t)t=0.$$
Moreover,
\begin{align}\label{pointsys1}
 a(|\nabla \bfu|)^2 {\rm {div}} ((\Delta \bfu)^T \nabla \bfu)
%
={\rm {div}}\big( a(|\nabla \bfu|)^2 (\Delta \bfu)^T\nabla \bfu\big)
%
- 
2
 a(|\nabla \bfu|)   a'(|\nabla \bfu|) \Delta \bfu \cdot \nabla \bfu (\nabla |\nabla \bfu|)^T,
%
%
\end{align}
and 
\begin{align}\label{pointsys2}
\tfrac 12   a(|\nabla \bfu|)^2    {\rm {div}}\big(\nabla |\nabla \bfu|^2\big)
%
%
 = 
 \tfrac 12{\rm {div}}\big(  a(|\nabla \bfu|)^2\,\nabla |\nabla \bfu|^2\big)
%
%
-   2a(|\nabla \bfu|)   a'(|\nabla \bfu|)|\nabla \bfu||\nabla |\nabla \bfu||^2.
\end{align}
From equations \eqref{point1}--\eqref{pointsys2} one deduces that
\begin{align}\label{pointsys3neq0}
\big|  {\rm {\bf div}}  (a(|\nabla \bfu|)\nabla \bfu )\big|^2 &= 
{\rm { div}}\big(a(|\nabla \bfu|)^2(\Delta \bfu)^T \nabla \bfu\big) -
\tfrac 12 {\rm {div}}\big(a(|\nabla \bfu|)^2\,\nabla |\nabla \bfu|^2\big)
\\ \nonumber & \quad
+
  a(|\nabla \bfu|)^2   |\nabla ^2 \bfu|^2 
 + a'(|\nabla \bfu|)^2 \big| \nabla \bfu (\nabla |\nabla \bfu|)^T \big|^2 
+    2a(|\nabla \bfu|)  a'(|\nabla \bfu|)|\nabla \bfu||\nabla |\nabla \bfu||^2.
%
%
%
\end{align}
If $\nabla \bfu =0$, then the last two addends on the right-hand side of equation \eqref{pointsys3neq0} vanish.   Hence,   equation \eqref{pointwiseq} follows. 
Assume next that $\nabla \bfu\neq 0$. 
Then,  from equation  \eqref{pointsys3neq0}  we obtain that
\begin{align*}
%
\big| {\rm {\bf div}}   (a(|\nabla \bfu|)\nabla \bfu )\big|^2
%
%
%
%
%
%
  & ={\rm {div}}\big(a(|\nabla \bfu|)^2 (\Delta \bfu)^T \nabla \bfu\big) -
\tfrac 12 {\rm {div}}\big(a(|\nabla \bfu|)^2\,\nabla |\nabla \bfu|^2\big)
%
%
%
%
%
%
%
\\ \nonumber & \quad \ + 
a(|\nabla \bfu|)^2  \Bigg[ |\nabla ^2 \bfu|^2 
+\bigg(\frac{a'(|\nabla \bfu|) |\nabla \bfu|}{a(|\nabla \bfu|)}\bigg)^2 \bigg| \frac{\nabla \bfu}{|\nabla \bfu|} (\nabla |\nabla \bfu|)^T \Bigg|^2 
 +2 \frac{a'(|\nabla \bfu|) |\nabla \bfu|}{a(|\nabla \bfu|)}|\nabla |\nabla \bfu||^2\Bigg].
\end{align*} 
The proof of equation \eqref{pointwiseq} is complete.
\end{proof}

Having identity \eqref{pointwiseq} at our disposal, the point is now to derive a sharp lower bound for the second addend on its right-hand side. This will be accomplished via Lemma \ref{cor:pmin} below. Its proof requires  a delicate analysis of the quadratic form depending on the entries of the Hessian matrix $\nabla^2 \bfu$ which appears  in square brackets in the expression to be bounded. This analysis relies upon some critical linear-algebraic steps that are presented in the next three lemmas. 
\\
In what follows,  $\setR^{n\times n}_{\sym}$ denotes the space of
symmetric matrices in $\setR^{n \times n}$. The dot $\lq\lq \, \cdot \, "$ is employed to denote  scalar product of vectors or matrices, and the symbol $\lq\lq \otimes "$  for tensor product of vectors. Also,   $I$ stands for the identity matrix in $\setR^{n \times n}$.

\begin{lemma}
  \label{lem:cianchi-mazya}
  Let $\omega \in \rn$ be such that $\abs{\omega} =1$. Then, 
  \begin{align}\label{may8}
    \abs{H\omega}^2-\tfrac 12 \abs{\omega \cdot H\omega}^2-\tfrac 12 \abs{H}^2
    =- \tfrac 12 \abs{H_{\omega ^\perp}}^2
  \end{align}
for every $H \in \setR^{n \times
    n}_{\sym}$,
  where $H_{\omega ^\perp} = (I - \omega \otimes \omega) H (I -
  \omega \otimes \omega)$.
\end{lemma}
\begin{proof}
  Let $\{e_1, \dots, e_n\}$ denote the canonical basis in $\rn$ and let $\{\theta_1, \dots, \theta_n\}$ be an orthonormal basis
  of~$\rn$ such that $\theta_1 = \omega$. Let $Q \in  \setR^{n \times n}$ be the matrix whose columns are $\theta_1, \dots, \theta_n$. Hence,  $\omega = Q e_1$. Next, let 
  $R=  Q^T H Q$.  Clearly, $R \in \setR^{n \times n}_{\sym}$.  Denote by $r_{ij}$ the entries of $R$. Computations show that
  \begin{align*}
 {\abs{H\omega}^2-\tfrac 12 \abs{\omega \cdot H\omega}^2 -\tfrac  12
    \abs{H}^2 } 
    &
   = \abs{R e_1}^2-\tfrac 12
      \abs{e_1 \cdot R e_1}^2-\tfrac 12 \abs{R}^2
      = 
  \sum_{i=1}^n \abs{r_{i1}}^2-\tfrac 12
      \abs{r_{11}}^2-\tfrac 12 \sum_{i,j=1}^n\abs{r_{ij}}^2 
      \\&
      =\tfrac 12\sum_{j=1}^n
      \abs{r_{1j}}^2 + \tfrac 12 \sum_{i=1}^n \abs{r_{i1}}^2-\tfrac 12
      \abs{r_{11}}^2-\tfrac 12 \sum_{i,j=1}^n\abs{r_{ij}}^2
=-\tfrac 12 \sum_{i,j \ge 2}\abs{r_{ij}}^2
    \\
    &= -\tfrac 12 \abs{(I - e_1 \otimes e_1) R (I- e_1 \otimes e_1)}^2
   = -\tfrac 12 \abs{(I - \omega
      \otimes \omega) H (I - \omega
      \otimes \omega)}^2.
  \end{align*}
 Hence, equation \eqref{may8} follows.
\end{proof}

Given  a vector $\omega \in \rn$, define the set
\begin{align*}
  E(\omega) = \bigset{ H \omega\,:\, \text{$H \in \setR^{n \times
              n}_{\sym},
              \abs{H} \leq 1$}}.
\end{align*}
 It is easily verified that  $E(\omega)$ is a convex set in $\rn$ for every  $\omega \in \rn$. Lemma  \ref{lem:Homega-ellipsoid} below tells us that, in fact, $E(\omega)$ is an ellipsoid, centered at $0$ (which reduces to $\{0\}$ if $\omega =0$). This assertion will be verified by showing that, for each $\omega \in \rn$, there exists a positive definite matrix $W \in \setR^{n \times n}_{\sym}$ such that $E(\omega)$  agrees with the ellipsoid 
\begin{align}\label{may1}
  F(W) = \bigset{
         x\in \rn:\,  x\cdot W^{-1} x\leq 1},
\end{align}
where $W^{-1}$ stands for the inverse of $W$.
This is the content of Lemma \ref{lem:Homega-ellipsoid} below. In   its proof, we shall make use of the 
alternative representation
\begin{align}\label{may2}
  F(W) =
         \bigset{x\in \rn:\, y \cdot x \leq \sqrt{y \cdot Wy} \quad \text{for
         every $y \in \setR^n$}},
\end{align}
which follows, for instance, via a maximization argument for the ratio of the two sides of the inequality in \eqref{may2} for each given $x\in \rn$.
\\ Also, observe  that, as a consequence of equation \eqref{may2}, 
\begin{align}
  \label{eq:est-ellipsoid}
  \abs{x} =  x \cdot \widehat{x}
\leq \sqrt{ \widehat{x} \cdot W \widehat{x}}
                                   \qquad \text{for every  $x \in F(W)\setminus \{0\}$.} 
\end{align}
Here, and in what follows, we adopt the notation
$$\widehat x = \frac x{|x|} \qquad \hbox{for $x \in \rn\setminus \{0\}$.}$$

\begin{lemma}
  \label{lem:Homega-ellipsoid}
Given $\omega \in \rn$, let $W(\omega) \in  \setR^{n \times n}_{\sym}$ be defined as 
  \begin{align}\label{may3}
    W(\omega) = \tfrac 12 \big( \abs{\omega}^2 I + \omega
                \otimes \omega \big).
  \end{align}
Then $W(\omega)$ is positive definite, and 
  \begin{align}\label{may4}
    E(\omega) = F(W(\omega)).
  \end{align}
  In particular,
  \begin{align}\label{may5}
    H \omega \in \abs{H}\,F\big( W(\omega)\big) \quad \text{for every  $\omega \in \rn$ and  $H \in \setR^{n \times n}_{\sym}$.} 
  \end{align}
\end{lemma}
\begin{proof} Equation \eqref{may4} trivially holds if $\omega =0$. Thus, 
  by a   scaling argument,   it  suffices to consider the
  case when $\abs{\omega}=1$.  
  We begin showing that  $E(\omega) \subset F(W(\omega))$. 
One can verify that,   since   $\abs{\omega}=1$, 
  \begin{align}\label{may6}
    W(\omega)^{-1} = 2I - \omega \otimes
             \omega.
  \end{align}
Let  $H \in \setR^{n \times n}_{\sym}$ be such that $\abs{H} \leq 1$. 
Owing to equation \eqref{may6} and to
  Lemma~\ref{lem:cianchi-mazya},  
  \begin{align}\label{may7}
    H\omega \cdot W(\omega)^{-1} H\omega
    = 2 \abs{H \omega}^2 -
      \bigabs{\omega \cdot H \omega}^2 \leq \abs{H}^2 \leq 1.
  \end{align}
  This shows that $H \omega \in F(W(\omega))$.  The
  inclusion $E(\omega) \subset F(W(\omega))$ is thus established .
\\ Let  us next prove that  $F(W(\omega)) \subset E(\omega)$. Let
  $x \in F(W(\omega))$. We have to detect a matrix ~$H\in  \setR^{n \times n}_{\sym}$ such that $\abs{H} \leq 1$
  and $x = H \omega$. To this purpose,  consider the decomposition
  $x=t \omega + s \omega^\perp$,  for suitable $s, t \in \setR$, where  $\omega^\perp\perp \omega$ and
  $\abs{\omega^\perp}=1$. Since $x \in F(W(\omega))$, one has that
  $x \cdot W(\omega) ^{-1} x \leq 1$. Furthermore,
  \begin{align*}
    x \cdot W(\omega)^{-1}x =(t \omega + s \omega^\perp) \cdot (2 I - \omega
                         \otimes \omega) (t \omega + s \omega^\perp)
   = 2(t^2 + s^2) - t^2 = t^2 + 2s^2.
  \end{align*}
  Hence, $t^2 + 2s^2 \leq 1$.  We claim that the matrix $H$ defined as
$H = t\,\omega \otimes \omega + s\,(\omega^\perp \otimes \omega +
  \omega \otimes \omega^\perp)$, has the desired properties. Indeed, 
$H \in  \setR^{n \times n}_{\sym}$,
  \begin{align*}
    \abs{H}^2 &= {\rm tr}(H^T H) = t^2 + 2s^2 \leq 1
\\
    H \omega &= t \omega + s \omega^\perp = x.
  \end{align*}
  This proves that $x \in E(\omega)$. The inclusion $F(W(\omega)) \subset
  E(\omega)$ hence follows. 
\end{proof}
In view of the statement of the next lemma, we introduce the following notation. Given $N$ vectors $\omega^\alpha \in \rn$ and $N$ matrices $H^\alpha \in  \setR^{n \times n}_{\sym}$, with $\alpha =1, \dots N$, we set
\begin{align}\label{may10}
    J  &= \biggabs{ \sum_{\alpha =1}^N H^\alpha \omega^\alpha}^2, &
    J_0 &= \sum_{\alpha =1}^N\biggabs{ \omega^\alpha \cdot  \sum_{\beta =1}^N H^\beta
      \omega^\beta}^2, &
    J_1  &= \sum_{\alpha =1}^N \abs{H^\alpha}^2.
  \end{align}

\begin{lemma}
  \label{thm:max-defab}
  Let $N \geq 2$, $0 \leq \delta \leq \frac 12$ and
  $\delta+\sigma \geq 1$.
  Assume that   the vectors $\omega^\alpha \in \rn$ and  the matrices $H^\alpha \in  \setR^{n \times n}_{\sym}$, with $\alpha =1, \dots N$, satisfy the following constraints:
  \begin{align}\label{may11}
\sum_{\alpha=1}^N \abs{\omega^\alpha }^2 \leq 1,
\end{align}
\begin{align}\label{may12}
    \sum_{\alpha=1}^N \abs{H^\alpha}^2 &\leq 1.
  \end{align}
 Then,
  \begin{align}\label{may13}
    J - \delta J_0 - \sigma J_1
    &\leq
      \begin{cases}
        0 &\qquad \text{if $\delta \in [0,\frac 13]$},
        \\
        \max \Bigset{0,\frac{(\delta+1)^2}{8\delta} - \sigma} &\qquad \text{if
          $\delta \in (\frac 13,\frac 12]$}.
      \end{cases}
  \end{align}
\end{lemma}
\begin{proof}
 Given $\delta$ and $\sigma$ as in the statement, set
  \begin{align*}
    \mathcal{D}_{\delta,\sigma} = J - \delta J_0 - \sigma J_1.
  \end{align*}
  The quantities $J_0$, $J$ and $J_1$ are 1-homogeneous
  with respect to the quantity $\sum_{j=1}^N \abs{H_j}^2$. Moreover,   inequality  \eqref{may13} trivially holds if the latter quantity vanishes. Thereby, it suffices to prove this inequality under the assumption that $ \sum_{j=1}^N \abs{H_j}^2=1$, namely that
  \begin{align}\label{may14}
    J_1 = 1.
  \end{align}
On setting $\zeta = \sum_{\alpha=1}^N H^\alpha \omega^\alpha$, one has that
%
  \begin{align*}
    J = \abs{\zeta}^2  \qquad \text{and} \qquad
    J_0 = \sum_{\alpha=1}^N \abs{ \omega^\alpha \cdot \zeta}^2.
  \end{align*}
  Therefore,
  \begin{align}\label{may15}
    J_0 \leq \abs{\zeta}^2 \sum_{\alpha=1}^N \abs{\omega^\alpha}^2  \leq \abs{\zeta}^2 = J.
  \end{align}
  Owing to Lemma~\ref{lem:Homega-ellipsoid},  
    $$
H^\alpha \omega^\alpha \in \abs{H^\alpha} 
                   F(W^\alpha) 
 $$
for $\alpha =1, \dots,  N$,
  where $W^\alpha = \abs{\omega^\alpha}^2 \frac 12 (\identity +
  \widehat{\omega^\alpha} \otimes \widehat{\omega^\alpha})$.
Thus, by equations \eqref{may5} and \eqref{eq:est-ellipsoid},
  \begin{align}\label{may17}
 H^\alpha \omega^\alpha \cdot \widehat{\zeta}
    &\leq \abs{H^\alpha } \sqrt{\widehat{\zeta}\cdot W^\alpha 
      \widehat{\zeta}} = \abs{H^\alpha }  \abs{\omega^\alpha} \sqrt{ \tfrac 12 +
      \tfrac 12 \abs{\widehat{\omega^\alpha} \cdot \widehat{\zeta}}^2}
  \end{align}
for $\alpha =1, \dots,  N$.
  Since
  \begin{align*}
    \zeta &= ( \zeta \cdot \widehat{\zeta})\widehat{\zeta} = \sum_{\alpha =1}^N
            (H^\alpha \omega^\alpha \cdot \widehat{\zeta}) \widehat \zeta, 
  \end{align*}
  equation \eqref{may17} implies that
  \begin{align*}
    \abs{\zeta} &\leq \sum_{\alpha =1}^N \bigabs{H^\alpha \omega^\alpha \cdot \widehat{\zeta}}
                  \leq \sum_{\alpha=1}^N \abs{H^\alpha }  \abs{\omega^\alpha} \sqrt{ \tfrac 12 +
      \tfrac 12 \abs{\widehat{\omega^\alpha} \cdot \widehat{\zeta}}^2}.
  \end{align*}
  Hence,
  \begin{align}\label{may18}
    \abs{\zeta}^2 &\leq \tfrac 12 
                    \bigg(\sum_{\alpha=1}^N \abs{H^\alpha }  \abs{\omega^\alpha} \sqrt{ \tfrac 12 +
      \tfrac 12 \abs{\widehat{\omega^\alpha} \cdot \widehat{\zeta}}^2}     \bigg)^2. 
  \end{align}
 On setting $ \widehat{J_0}
    = \sum_{\alpha =1}^N \abs{\smash{\omega^\alpha \cdot \widehat{\zeta}}}^2$, we obtain that
  \begin{align*}
    \widehat{J_0} 
      = \sum_{\alpha=1}^N \abs{\omega^\alpha}^2 \abs{\smash{\widehat{\omega^\alpha}\cdot\widehat{\zeta}}}^2 \qquad \text{and} \qquad
    J_0 = \abs{\zeta}^2 \widehat{J_0}.
  \end{align*}
 Note that   $\widehat{J_0} \leq 1$, inasmuch as $J_0 \leq J= \abs{\zeta}^2$.
  Moreover, by equation \eqref{may14},
  \begin{align}\label{may19}
    \mathcal{D}_{\delta,\sigma} &= J - \delta J_0 - \sigma= \abs{\zeta}^2
                  \big(1 - \delta
                  \widehat{J_0}\big) - \sigma.
  \end{align}
 From inequalities \eqref{may18} and \eqref{may19} we deduce  that
  \begin{align}\label{may20}
    \mathcal{D}_{\delta,\sigma}
    &\leq \tfrac 12 
                    \bigg(\sum_{\alpha=1}^N \abs{H^\alpha }  \abs{\omega^\alpha} \sqrt{ \tfrac 12 +
      \tfrac 12 \abs{\widehat{\omega^\alpha} \cdot \widehat{\zeta}}^2}     \bigg)^2
  \Big(1 - \delta \sum_{\alpha=1}^N  \abs{\omega^\alpha}^2 \abs{\smash{\widehat{\omega^\alpha}\cdot\widehat{\zeta}}}^2
      \Big) - \sigma.
  \end{align}
Next, define the function 
  with $g\,:\ [0,1]^N \times [0,1]^N \times [0,1]^N \to \setR$ as
  \begin{align}
    \label{eq:def-g}
    \begin{aligned}
      g(h,s,t) 
      & = \tfrac 12 \Big( \sum_{\alpha=1}^N h_\alpha t_\alpha\sqrt{1 +
        s_\alpha^2}  \Big)^2  \bigg(1 - \delta
      \Big( \sum_{\alpha=1}^N t_\alpha^2 s_\alpha^2 \Big) \bigg) - \sigma 
      \\
     & \quad \text{for $(h,s,t) \in [0,1]^N \times [0,1]^N \times [0,1]^N$,}
    \end{aligned}
  \end{align}
where $h=(h_1, \dots , h_N)$, $s=(s_1, \dots, s_N)$ and $t=(t_1, \dots, t_N)$. Inequality \eqref{may20} then takes the form
  \begin{align*}
    \mathcal{D}_{\delta,\sigma}
    &\leq g((|H^1|, \dots, |H^N|), (|\omega^1|, \dots , |\omega^N|), (|\abs{\widehat{\omega^1} \cdot \widehat{\zeta}}, \dots , \abs{\widehat{\omega^N} \cdot \widehat{\zeta}})).
  \end{align*}
Our purpose is now to
 maximize the function $g$   under the  constraints
  \begin{align}\label{may21}
    \sum_{\alpha=1}^N t_\alpha^2
    & \leq 1,
    &
    \sum_{\alpha=1}^N h_\alpha^2 &= 1.
  \end{align}
We claim that the maximum of $g$ can   only be attained if
  $\sum_{\alpha=1}^N t_\alpha^2 = 1$.  To verify this claim, it suffices to show that
  \begin{align}
    \label{eq:g-hom-t}
    g(h, s, 
    \tau t) &\leq g(h,s,t) \qquad \text{for every $(h,s,t) \in [0,1]^N \times [0,1]^N \times [0,1]^N$ and $\tau \in [0,1]$}.
  \end{align}
 Plainly,
  \begin{align*}
    g(h,s,\tau t)
    =  \tfrac 12 \tau ^2\Big( \sum_{\alpha=1}^N h_\alpha t_\alpha\sqrt{1 +
        s_\alpha^2}  \Big)^2  \bigg(1 - \tau ^2\delta
      \Big( \sum_{\alpha=1}^N t_\alpha^2 s_\alpha^2 \Big) \bigg) - \sigma
%
%
%
  \end{align*}
for  $(h,s,t) \in [0,1]^N \times [0,1]^N \times [0,1]^N$ and $\tau \in [0,1]$.
  Note that
  \begin{align}
    \label{eq:1}
    0 &\leq \delta      \Big( \sum_{\alpha=1}^N t_\alpha^2 s_\alpha^2 \Big) 
    \leq \delta       \Big( \sum_{\alpha=1}^N t_\alpha^2  \Big)  = \delta \leq
        \tfrac 12.
  \end{align}
  Thus, for each fixed $(h,s,t) \in [0,1]^n \times [0,1]^n \times [0,1]^n$, we have that
  \begin{align}\label{may23}
    g(h, s, \tau t) = c_1 \tau (1- c_2 \tau) - \beta  \qquad \text{for $\tau \in [0,1]$,}
  \end{align}
for suitable constants 
              $c_1 \geq 0$ and
              $0\leq c_2
              \leq \tfrac 12$, depending on  $(h,s,t)$.   Since the polynomial on the right-hand side of equation \eqref{may23} is increasing for $\tau \in [0,1]$,
inequality \eqref{eq:g-hom-t} follows.
As a consequence, constraints \eqref{may21} can be equivalently replaced by
  \begin{align}\label{may24}
    \sum_{\alpha=1}^N t_\alpha ^2 &=1 \qquad \text{and} \qquad    \sum_{\alpha=1}^N h_\alpha^2 =1.
  \end{align}
  Let us maximize the function $g(h,s,t)$ with respect to~$h$, under the
  constraint $\sum_{\alpha=1}^N h_\alpha^2 = 1$.  Let $(h_1, \dots , h_N)$ be any point where the maximum is attained. Then,   there exists  a
  Langrange multiplier~$\lambda \in \setR$ such that 
  \begin{align}
    \label{eq:lagr-hj}
   t_\alpha
    \sqrt{1  + s_\alpha^2}   \bigg(  \sum_{\gamma=1}^N h_\gamma t_\gamma
    \sqrt{1  + s_\gamma^2}
    \bigg) \bigg(1 - \delta \Big( \sum_{\gamma=1}^N t_\gamma^2 s_\gamma^2 \Big)
    \bigg) &= 2\lambda h_\alpha \quad \hbox{for $\alpha =1, \dots, N$.}
  \end{align}
  Multiplying through equation ~\eqref{eq:lagr-hj} by $h_\beta$, and then
  subtracting  equation ~\eqref{eq:lagr-hj}, with $\alpha$ replaced by~$\beta$, multiplied 
  by~$h_\alpha$ yield 
  \begin{align}
    \label{eq:3}
  \bigg(  \sum_{\gamma=1}^N h_\gamma t_\gamma
    \sqrt{1  + s_\gamma^2}
    \bigg)  \bigg(1 - \delta \Big( \sum_{\gamma=1}^n t_\gamma^2 s_\gamma^2 \Big) \bigg)
    \Big(
    h_\beta t_\alpha
    \sqrt{1  + s_\alpha^2}  -     h_\alpha t_\beta
    \sqrt{1  + s_\beta^2}  \Big) &= 0
  \end{align}
for ~$\alpha, \beta = 1,\dots, N$. Owing to
 equation \eqref{eq:1}, we have that $\big(1 - \delta \big( \sum_{\gamma=1}^n t_\gamma^2 s_\gamma^2 \big) \big) \geq \frac 12$. Next, if $ \sum_{\gamma=1}^N h_\gamma t_\gamma
    \sqrt{1  + s_\gamma^2}=0$, 
then $h_1t_1 = \dots = h_N t_N = 0$, whence  $\mathcal{D}_{\delta,\sigma} = -\sigma \leq 0$, and inequality \eqref{may13} holds trivially.
Therefore, 
  we may assume that $ \sum_{\gamma=1}^N h_\gamma t_\gamma
    \sqrt{1  + s_\gamma^2} > 0$ in what follows. Under this assumption, 
  equation \eqref{eq:3} tells us that  
  \begin{align}
    \label{eq:2}
      h_\beta t_\alpha
    \sqrt{1  + s_\alpha^2}  &=  h_\alpha t_\beta
    \sqrt{1  + s_\beta^2}
  \end{align}
for ~$\alpha, \beta = 1,\dots, N$. Combining equations \eqref{may24} and \eqref{eq:2}
  yields
  \begin{align}
    \label{eq:4}
    \begin{aligned}
      t_\alpha^2(1+s_\alpha^2) &= t_\alpha^2(1+s_\alpha^2) \sum_{\beta=1}^N h_\beta^2
      =
      h_\alpha^2 \sum_{\beta =1}^N t_\beta^2(1+s_\beta^2) = h_\alpha^2 \bigg( 1 + \sum_{\beta=1}^N
      t_\beta^2s_\beta^2  \bigg)
    \end{aligned}
  \end{align}
for ~$\alpha = 1,\dots, N$.
  Hence,
  \begin{align}\label{may25}
      h_\alpha t_\alpha \sqrt{1+s_\alpha ^2}
                     &= h_\alpha^2 \sqrt{1 + \sum_{\beta=1}^N t_\beta ^2s_\beta^2}
  \end{align}
for $\alpha = 1,\dots, N$. From
equations \eqref{eq:def-g}, \eqref{may25} and \eqref{may24} we deduce that
  \begin{align*}
    g(h,s,t) &\leq \tfrac 12 \left( \sum_{\alpha=1}^N h_\alpha^2
               \sqrt{1+\sum_{\beta=1}^N t_\beta^2s_\beta^2}   \right)^2  \bigg(1 - \delta
               \Big( \sum_{\alpha =1}^N t_\alpha^2 s_\alpha^2 \Big) \bigg) - \sigma
    \\
             &= \tfrac 12 \bigg( 1+\sum_{\beta=1}^N t_\beta^2s_\beta^2 \bigg)   \bigg(1 - \delta
               \Big( \sum_{\alpha=1}^N t_\alpha^2 s_\alpha^2 \Big) \bigg) - \sigma
    = \psi\bigg( \sum_{\alpha=1}^N t_\alpha^2 s_\alpha^2 \bigg),
  \end{align*}
  where $\psi : [0,1] \to \mathbb R$ is the function defined as
  \begin{align*}
    \psi(r) =  \tfrac 12 (1 +r)
              \big(1
              - \delta r \big) - \sigma \quad \text{for $r \in \mathbb R$.}
  \end{align*}
  Set $\rho = \sum_{j=\alpha}^N t_\alpha^2 s_\alpha^2$, and notice that $\rho \in [0,1]$, since $0 \leq
  \sum_{\alpha=1}^N t_\alpha^2 s_\alpha^2 \leq \sum_{\alpha=1}^N t_\alpha^2=1$.
  Thereby, the maximum of the function $g$ on $[0,1]^N \times [0,1]^N \times [0,1]^N$ under constraints \eqref{may24} agrees with the 
 maximum of the function $\psi$ on $[0,1]$. It is easily verified that, if  $\delta \in [0, \frac 13]$, then $\max_{r \in [0,1]}\psi (r) = \psi (1)$. Hence, since we are assuming that $\delta + \sigma \geq 1$,
\begin{align*}
      \mathcal{D}_{\delta,\sigma} \leq \psi(1)  =
                                     1-\delta-\sigma \leq 0.
    \end{align*}
On the other hand, if  $\delta \in (\frac 13, \frac 12]$, then $\max_{r \in [0,1]}\psi (r) = \psi (\frac{1-\delta}{2\delta})$. Therefore, 
\begin{align*}
      \mathcal{D}_{\delta,\sigma}
      &\leq     \psi\Big( \frac{1-\delta}{2\delta} \Big)
        = \frac{(\delta+1)^2}{8\delta} - \sigma. 
    \end{align*}
The proof of inequality \eqref{may13} is complete.
\end{proof}

\begin{lemma}
  \label{cor:pmin} Let $n$, $N$, $\Omega$ and $\bfu$ be as in Theorem \ref{lemma1}.
  Given $p\geq 1$,
let $\kappa _N (p)$ be the constant defined by \eqref{kappa1}--\eqref{cp}. Then
  \begin{align}\label{eq:est_gamma}
    \abs{\nabla^2 \bfu}^2
    +  2(p-2)\bigabs{\nabla \abs{\nabla \bfu}}^2
    +  (p-2)^2 \biggabs{ \frac{\nabla \bfu}{\abs{\nabla \bfu}}
    (\nabla \abs{\nabla \bfu})^T}^2 &\geq \kappa_N (p)
                                   \abs{\nabla^2 \bfu}^2
  \end{align}
in $\{\nabla \bfu \neq 0\}$. Moreover, the constant $\kappa_N (p)$  is sharp in \eqref{eq:est_gamma}.
\end{lemma}
\begin{proof} \emph{Case $N=1$.} Inequality \eqref{eq:est_gamma} trivially holds if $p\geq 2$.  Let us focus on the case when $1\leq p < 2$.
Notice that, on setting
$$\omega = \frac{(\nabla u)^T}{|\nabla u|}\in \rn \quad \text{and} \quad H= \nabla ^2 u \in \setR^{n\times n}_{\rm sym}$$
 at any point  in $\{\nabla u \neq 0\}$, we have that
\begin{equation*}
\abs{H\omega}^2= \bigabs{\nabla \abs{\nabla u}}^2,  \quad \abs{\omega \cdot H\omega}^2= \biggabs{ \frac{\nabla u}{\abs{\nabla u}}
    (\nabla \abs{\nabla u})^T}^2,  \quad\abs{H}^2=  \abs{\nabla^2 u}^2.
\end{equation*}
Therefore, by  equation \eqref{may8},  
  \begin{align*}
    \bigabs{\nabla \abs{\nabla u}}^2
    &\leq 
      \tfrac 12 \biggabs{ \frac{\nabla u}{\abs{\nabla u}}
      (\nabla \abs{\nabla u})^T}^2 + \tfrac 12 \abs{\nabla^2 u}^2.
  \end{align*}
Consequently, the following chain holds:
  \begin{align*}
    \lefteqn{\abs{\nabla^2 u}^2
    +  2(p-2)\bigabs{\nabla \abs{\nabla u}}^2
    +  (p-2)^2 \biggabs{ \frac{\nabla u}{\abs{\nabla u}}
    (\nabla \abs{\nabla u})^T}^2} \qquad
    &
    \\
    &
\geq \big (1 + (p-2) \big)
      \abs{\nabla^2 u}^2
      +  \big( (p-2) + (p-2)^2\big) \biggabs{ \frac{\nabla u}{\abs{\nabla u}}
      (\nabla \abs{\nabla u})^T}^2 
    \\
    &\geq (p-1)
      \abs{\nabla^2 u}^2
      +  (p-1)(p-2) \biggabs{ \frac{\nabla u}{\abs{\nabla u}}
      (\nabla \abs{\nabla u})^T}^2 
    \\
    &\geq \big((p-1) + (p-1)(p-2)\big)
      \abs{\nabla^2 u}^2
    \\
    &= (p-1)^2
      \abs{\nabla^2 u}^2.
  \end{align*}
  Hence, inequality \eqref{eq:est_gamma} follows.
\\ As far as the sharpness of the constant is concerned, 
if $p\geq 2$,  consider the function  $u: \setR^n\setminus \set{0} \to \setR$ given by
$$u (x)= |x|  \quad \text{for $x \in \setR^n\setminus \set{0} $.}$$ 
Since  $\nabla \abs{\nabla u} = 0$,   equality holds in \eqref{eq:est_gamma}  for every $x \in \setR^n\setminus \set{0}$.
On the other hand, if $p \in [1, 2)$, consider the function $u: \setR^n\to \setR$ defined as 
$$ u (x)= \tfrac 12 x_1^2  \quad \text{for $x \in \setR^n$.}$$ 
One has that
$$ \abs{\nabla^2 u}^2
    =\bigabs{\nabla \abs{\nabla u}}^2
   = \biggabs{ \frac{\nabla u}{\abs{\nabla u}}
    (\nabla \abs{\nabla u})^T}^2 =1 \quad \text{in $\rn$.}$$
Hence, equality  holds in \eqref{eq:est_gamma} for every $x \in \setR^n\setminus \set{0}$.

\smallskip
\par\noindent
\emph{Case  $N \geq 2$.}
It suffices to prove that inequality \eqref{eq:est_gamma} holds at every point $x\in \{\nabla \bfu \neq 0\}$ under the assumption that $\abs{\nabla^2 \bfu(x)}$  equals either $0$ or $1$. Indeed, if $\abs{\nabla^2 \bfu(x)}\neq 0$ at some  point $x$, then the function given by $\overline \bfu   = \frac{\bfu }{\abs{\nabla^2 \bfu(x)}}$ fulfills $\abs{\nabla^2 \overline \bfu(x)}=1$. Hence, inequality \eqref{eq:est_gamma} for $\bfu$ at the point $x$ follows from the same inequality applied to $\overline u$. 
%
 \\ If~$p\geq 2$, inequality \eqref{eq:est_gamma}  holds trivially.  Thus, we may focus on the case when  $p \in [1,2)$.
In this case, we make use of  Lemma~\ref{thm:max-defab}.
Define
$$\omega^\alpha = \frac{\nabla \bfu^\alpha}{|\nabla \bfu|}\in \rn \quad \text{and} \quad H^\alpha= \nabla ^2 \bfu^\alpha \in \setR^{n\times n}_{\rm sym}$$
for $\alpha =1, \dots , N$, at any point  in $\{\nabla \bfu \neq 0\}$. In particular, assumptions \eqref{may11} and \eqref{may12} are satisfied with this choice. 
Computations show that
\begin{equation}\label{may33}
J= \bigabs{\nabla \abs{\nabla \bfu}}^2, \quad J_0= \biggabs{ \frac{\nabla \bfu}{\abs{\nabla \bfu}}
    (\nabla \abs{\nabla \bfu})^T}^2, \quad J_1 =  \abs{\nabla^2 \bfu}^2,
\end{equation}
where $J$, $J_0$ and $J_1$ are defined as in \eqref{may10}.
\\ 
  Next, let~$\delta= \frac{2-p}{2}$. Notice  that $\delta \in   [0, \frac 12]$, and that
$\delta  \in(0, \frac 13]$ if and only if $p \in [\frac 43,2)$.  We next choose $\sigma = \frac p2$ if $p \in [\frac 43,2)$, and $\sigma= \frac{(\delta+1)^2}{8\delta}= \frac 1{16} \frac{(4-p)^2}{2-p}$ if $p\in [1, \frac 43)$. Observe that $\delta +\sigma =1$ in the former case, and 
$\delta +\sigma > 1$ in the latter.  
Thus,   the assumptions on $\delta$ and $\sigma$ of  Theorem~\ref{thm:max-defab} are fulfilled. Furthermore, by our choice of $\sigma$, the maximum on right-hand side of inequality \eqref{may13} equals $0$ when $\delta  >\frac 13$, namely when $p\in  [1, \frac 43)$.
 From inequality \eqref{may13} we infer that
\begin{align*}
    J &\leq \tfrac{2-p}2 J_0 + \sigma J_1.
  \end{align*}
This inequality is equivalent to
$$J_1 + 2(p-2)J + (p-2)^2 J_0 \geq (1- \sigma 2(2-p)) J_1.$$
Since $1- \sigma 2(2-p) = \mathcal K(p)$, inequality \eqref{eq:est_gamma} follows.
\\ In order to prove the sharpness of the constant $\mathcal K(p)$, let us distinguish the cases when $p\geq 2$, $p \in [\frac 43, 2)$ and $p\in [1, \frac 43)$.
\\ If $p\geq 2$, 
consider the function  $\bfu: \setR^n\setminus \set{0} \to \setR^N$ given by
$$ \bfu (x)= (|x|, 0, \dots ,0) \quad \text{for $x \in \setR^n\setminus \set{0} $.}$$ 
Since  $\nabla \abs{\nabla \bfu} = 0$,   equality holds in \eqref{eq:est_gamma} for every $x \in \setR^n\setminus \set{0}$.
\\ If $p \in [\frac 43, 2)$, 
consider the function $\bfu: \setR^n\to \setR^N$ defined as 
$$ \bfu (x)= (\tfrac 12 x_1^2, 0, \dots ,0) \quad \text{for $x \in \setR^n$.}$$ 
One has that
$$ \abs{\nabla^2 \bfu}^2
    =\bigabs{\nabla \abs{\nabla \bfu}}^2
   = \biggabs{ \frac{\nabla \bfu}{\abs{\nabla \bfu}}
    (\nabla \abs{\nabla \bfu})^T}^2 =1 \quad \text{in $\rn$.}$$
Thus, equality  holds in \eqref{eq:est_gamma} for every $x \in \setR^n\setminus \set{0}$.
\\ If  $p \in [1, \frac 43)$, 
    set $r_0 = \frac p{2(2-p)}$. 
    Let $e_1, e_2$ denote the first two  vectors of the canonical base of $\rn$. Define
    \begin{align*}
      t_1 &= \sqrt{r_0}, &\qquad \omega^1 &= t_1 e_1,
      \\
      t_2 &= \sqrt{1-r_0}, &\qquad \omega^2 &= t_2 e_2,
      \\
      h_1 &= \sqrt{\frac{2r_0}{1+r_0}}, &\qquad H^1 &= h_1 e_1 \otimes e_1,
      \\
      h_2 &= \sqrt{\frac{1-r_0}{1+r_0}}, &\qquad H^2 &= h_2
                                                        \tfrac{1}{\sqrt{2}}
                                                        \big( e_1
                                                        \otimes e_2 +
                                                        e_2 \otimes e_1\big),
    \end{align*}
    and  $\omega^3 = \dots = \omega^N = 0$,  $H^3 = \dots = H^N = 0$.
    Then
    \begin{align}\label{may32}
      \sum_{\alpha=1}^N \abs{\omega^\alpha}^2 = \abs{\omega^1}^2 + \abs{\omega^2}^2 = 1.
\end{align}
Moreover,
    \begin{align}\label{may34}
      J_1 & =   \sum_{\alpha=1}^N \abs{H^\alpha}^2 =  \abs{H^1}^2 + \abs{H^2}^2= 1,
\\   \label{may35} J &=   \Biggabs{\sum_{\alpha=1}^N   H^\alpha \omega^\alpha }^2
= \abs{H^1 \omega^1 + H^2 \omega^2}^2 = \bigg|\Big( h_1 t_1 +
              \frac{1}{\sqrt{2}} h_2 t_2 \Big) e_1\bigg|^2 = \bigg|\sqrt{\frac{1+r_0}{2}} e_1\bigg|^2
                           =\frac{1+r_0}{2},
\\   \label{may36} J_0 & =  \sum_{\alpha=1}^N \abs{\ \omega^\alpha \cdot (H^1 \omega^1 + H^2 \omega^2)}^2
      = \bigabs{ \omega^ 1\cdot (H^1 \omega^1 + H^2 \omega^2) }^2 = \frac{r_0(1+r_0)}{2}.
    \end{align}
Now, let $\bfu : \rn \to \rN$ be a   polynomial of degree two such that $\nabla u^\alpha (0)^T = \omega^\alpha$ and
$\nabla^2u^\alpha = H^\alpha$ for $\alpha =1, \dots N$. Formulas \eqref{may33}, combined with \eqref{may34}--\eqref{may36}, tell us that
$$  \abs{\nabla^2 \bfu}^2
    +  2(p-2)\bigabs{\nabla \abs{\nabla \bfu}}^2
    +  (p-2)^2 \biggabs{ \frac{\nabla \bfu}{\abs{\nabla \bfu}}
    (\nabla \abs{\nabla \bfu})^T}^2 = 1- \tfrac{1}8(4-p)^2 = \kappa _N (p)
                                   \abs{\nabla^2 \bfu}^2 \quad \hbox{at $0$.}$$
Hence, equality holds in \eqref{eq:est_gamma} for   $x =0$. 
\end{proof}

We are now in a position to prove Theorem \ref{lemma1}.

\begin{proof}[Proof of Theorem \ref{lemma1}]
By Lemma \ref{cor:pmin}, applied with $p= Q_a(|\nabla \bfu|) +2$,
and  the monotonicity of the function $\kappa _N$ one has that
\begin{multline}\label{dec200}
a(|\nabla \bfu|)^2\Bigg[ |\nabla ^2 \bfu|^2 
 +2 Q_a(|\nabla \bfu|)|\nabla |\nabla \bfu||^2 +Q_a(|\nabla \bfu|)^2\bigg| \frac{\nabla \bfu}{|\nabla \bfu|} (\nabla |\nabla \bfu|)^T \bigg|^2 \Bigg]
  \\ \geq  \kappa_N \big(Q_a(|\nabla \bfu|)+2\big)
a(|\nabla \bfu|)^2  |\nabla ^2 \bfu|^2 \geq \kappa_N \big(i_a+2\big)
a(|\nabla \bfu|)^2  |\nabla ^2 \bfu|^2 \quad \text{in $\{\nabla \bfu \neq 0\}$.}
\end{multline}
Inequality \eqref{pointwise} holds at every point in the set  $\{\nabla \bfu \neq 0\}$, owing to equation \eqref{pointwiseq} and inequality \eqref{dec200}. It also trivially holds at every point in the set  $\{\nabla \bfu = 0\}$, since $\kappa_N \big(i_a+2\big)\leq 1$.
%
%
\\ In order to verify the optimality of the constant $\kappa _N(i_a+2)$ in inequality  \eqref{pointwise},
pick a function $\overline \bfu$ and a point $x_0$ from the  proof of Lemma \ref{cor:pmin} such that $\nabla \overline \bfu(x_0)\neq0$ and equality holds in inequality
 \eqref{eq:est_gamma} with $\bfu=\overline \bfu$ and $p=i_a +2$ at the point $x_0$. Namely,
\begin{align}\label{dec201}
 \abs{\nabla^2 \overline\bfu(x_0)}^2
    +  2i_a\bigabs{\nabla \abs{\nabla \overline\bfu}(x_0)}^2
    +  i_a^2 \biggabs{ \frac{\nabla \overline\bfu(x_0)}{\abs{\nabla \overline\bfu(x_0)}}
    (\nabla \abs{\nabla \overline\bfu}(x_0))^T}^2 = \kappa_N (i_a+2)
                                   \abs{\nabla^2 \overline\bfu(x_0)}^2.
\end{align}
By  the   the definition of the index $i_a$, given $\ep>0$ there exists $t_0 \in (0, \infty)$ such that  
\begin{align}\label{dec202}
  i_a \leq Q_a(t_0)  \leq  i_a + \ep.
\end{align}
Define , the function $\bfu = \frac{t_0 \overline\bfu}{|\nabla \overline\bfu(x_0)|}$, so that $|\nabla\bfu (0)|= t_0$.  From identity \eqref{pointwiseq}, equation \eqref{dec201} and inequality \eqref{dec202} we obtain that
\begin{align}\label{dec203}
&\frac{\big|{\rm {\bf div}} (a(|\nabla \bfu|)\nabla \bfu )\big|^2 - {\rm {div}} \Big[a(|\nabla \bfu|)^2  \Big(
(\Delta \bfu)^T \nabla \bfu
- \tfrac 12   \nabla |\nabla \bfu|^2\Big)\Big]}{a(|\nabla \bfu|)^{2}  |\nabla ^2 \bfu|^2 }\Bigg|_{x=x_0}
\\ \nonumber & \qquad = \frac{|\nabla ^2 \bfu(x_0)|^2 
 +2 Q_a(t_0)|\nabla |\nabla \bfu|(x_0)|^2 +Q_a(t_0)^2\bigg| \frac{\nabla \bfu(x_0)}{|\nabla \bfu(x_0)|} (\nabla |\nabla \bfu|(x_0))^T \bigg|^2}{ |\nabla ^2 \bfu(x_0)|^2 }
\\ \nonumber & \qquad = \frac{|\nabla ^2 \overline\bfu(x_0)|^2 
 +2 Q_a(t_0)|\nabla |\nabla \overline\bfu|(x_0)|^2 +Q_a(t_0)^2\bigg| \frac{\nabla\overline\bfu(x_0)}{|\nabla \overline\bfu(x_0)|} (\nabla |\nabla \overline\bfu|(x_0))^T \bigg|^2}{ |\nabla ^2 \overline\bfu(x_0)|^2 }
\\ \nonumber & \qquad \leq \frac{|\nabla ^2 \overline\bfu(x_0)|^2 
 +2 ( i_a + \ep)|\nabla |\nabla \overline\bfu|(x_0)|^2 + (i_a^2 + 2\ep |i_a|+ \ep^2))\bigg| \frac{\nabla \overline\bfu(x_0)}{|\nabla \overline\bfu(x_0)|} (\nabla |\nabla \overline\bfu|(x_0))^T \bigg|^2}{ |\nabla ^2 \overline\bfu(x_0)|^2 }
\\ \nonumber & \qquad = k_N(i_a+2)+ \frac{2 \ep|\nabla |\nabla \overline\bfu|(x_0)|^2 + (2\ep |i_a|+ \ep^2))\bigg| \frac{\nabla\overline\bfu(x_0)}{|\nabla \overline \bfu(x_0)|} (\nabla |\nabla \overline\bfu|(x_0))^T \bigg|^2}{ |\nabla ^2 \overline\bfu(x_0)|^2 }
\end{align}
Hence, the optimality   of the constant $\kappa _N(i_a+2)$ in inequality  \eqref{pointwise} follows, owing to the arbitrariness of $\ep$.
\end{proof}

\section{Function spaces}

An appropriate functional framework for the analysis of  solutions  to systems of the general form \eqref{system-a} is provided by the Orlicz-Sobolev spaces associated with the energy integral appearing in the functional \eqref{functional}. They consist in a generalization of the classical Sobolev spaces, where the role of powers in the definition of the norm is played by more general Young functions. Subsection \ref{Youngfunct} is devoted to  some basic definitions and properties of Young functions and  of Orlicz-Sobolev spaces. A  Poincar\'e type inequality for functions in these spaces of use for our purposes is established as well. In Subsection \ref{YoungB} we collect specific properties of the Young function (and of  perturbations of its) for the specific Orlicz-Sobolev ambient space associated with system  \eqref{localeq}.

\subsection{Young functions and Orlicz-Sobolev spaces}\label{Youngfunct}

   A Young  function $A : [0, \infty ) \to [0, \infty ]$ is
a convex function such that $A (0)=0$. 
The Young conjugate
of a Young function $A$ is the Young function $\widetilde A$ defined
as
\begin{equation*}
\widetilde A (t) = \sup \{st - A (s): s \geq 0\} \qquad \hbox{for $t
\geq 0$.}
\end{equation*}
A Young function (and, more generally, an increasing function) $A$
is said to belong to the class $\Delta _2$, or to satisfy the $\Delta_2$-condition, if there exists a
constant $c>1$ such that
\begin{equation}\label{delta2}
A (2t ) \leq c A (t) \qquad \hbox{for $t>0$.}
\end{equation}
Let   $i_A$ and $s_A$ be the indices associated with a continuously differentiable function $A$ as in \eqref{ia}, with $a$ replaced by $A$. Namely
\begin{equation}\label{indicesB}
i_A= \inf _{t >0} \frac{t A'(t)}{A(t)} \qquad \hbox{and} \qquad
s_A= \sup _{t >0} \frac{t A'(t)}{A(t)}.
\end{equation}
One has that $A \in \Delta _2$ if and only if $s_A<\infty$.  The constant $c$ in inequality \eqref{delta2} depends on $s_a$. Also, $\widetilde A \in \Delta _2$ if and only if $i_A>1$.

\medskip
\par
The Orlicz space $L^{A}(\Omega)$ is the Banach function space of those
real-valued  measurable functions $u : \Omega :\to \mathbb R$   whose Luxemburg
norm
\begin{equation*}
\|u\|_{L^A (\Omega)} = \inf \bigg\{\lambda >0: \int _\Omega A \bigg(\frac
{|u|}\lambda \bigg) \,d x \leq 1\bigg\}
\end{equation*}
is finite. The Orlicz space  $L^{A}(\Omega , \rN)$ of $\rN$-valued functions and the Orlicz space  $L^{A}(\Omega , \mathbb R^{N\times n})$ of $\mathbb R^{N\times n}$-valued functions are defined analogously.
\\ 
 The Orlicz-Sobolev space $W^{1, A
}(\Omega)$ is the Banach space
\begin{equation}
W^{1, A}(\Omega) =
 \{u \in L^A (\Omega):
\hbox{$u$ is   weakly differentiable in $\Omega$ and $\nabla u \in L^A (\Omega,  \rn)$} \}\,, \end{equation}
 and is equipped with the norm
$$\|u\|_{W^{1,A} (\Omega)} = \|u\|_{L^{A} (\Omega)}+ \|\nabla u\|_{L^A (\Omega,  \rn )}.$$
  The space $W^{1, A}_{\rm loc} (\Omega)$ is defined accordingly. By $W^{1,A} _0(\Omega)$ we denote the subspace  of $W^{1,A} (\Omega)$ of those functions in $W^{1,A} (\Omega)$ whose extension by $0$ outside $\Omega$ is weakly differentiable in the whole of $\rn$. The notation $(W^{1,A}_0 (\Omega))'$  stands for the dual of   $W^{1,A} _0(\Omega)$. 
If $\Omega$ has finite Lebesgue measure $|\Omega|$, then the functional $ \|\nabla u\|_{L^A (\Omega, \rn )}$ defines a norm in $W^{1,A} _0(\Omega )$ equivalent to $\|u\|_{W^{1,A} (\Omega)}$. 
\\ The space $C^\infty_0(\Omega)$ is dense in $W^{1,A} _0(\Omega)$  if $A\in \Delta_2$.  Moreover, $W^{1,A} _0(\Omega)$ is reflexive if   both $A\in \Delta_2$ and $ \widetilde A \in \Delta_2$, and hence if   $i_A>1$ and $s_A<\infty$.
\\The Orlicz-Sobolev space $W^{1,A} (\Omega,  \rN )$ of $\rN$-valued  functions,  its variants $W^{1,A}_{\rm loc} (\Omega,  \rN )$  and $W^{1,A}_0 (\Omega,  \rN )$, and  the space  $(W^{1,A}_0 (\Omega,  \rN ))'$ are defined analogously.
\par
If $|\Omega|<\infty$ and the Young function $A\in \Delta_2$, then the Poincar\'e type inequality
\begin{equation}\label{poinc}
\int_\Omega A(|u|)\, dx \leq c \int_\Omega A(|\nabla u|)\, dx 
\end{equation}
holds for some constant $c=c(n, |\Omega|, s_a)$ and for every function $u \in W^{1,A}_0(\Omega)$. Inequality \eqref{poinc} follows, for instance, from \cite[Lemma 3]{Talenti}.

In order to bound lower-order terms appearing in our global estimate, we also need a stronger, yet non-optimal, Sobolev-Poincar\'e type inequality for functions in $W^{1,A}_0(\Omega)$ with an Orlicz target space smaller than $L^A(\Omega)$.  This is the subject of Theorem \ref{sobolev} below, which generalizes  a version of the relevant inequality with optimal  Orlicz target space from  \cite{Ci_CPDE} (see also \cite{Ci_IUMJ} for an equivalent form). 
\\ Assume that the Young function $A$ and the number  $\sigma >1$  satisfy the conditions
\begin{equation}\label{conv0sig}
\int _0 \bigg(\frac t{A(t)}\bigg)^{\frac 1{\sigma -1}}\,dt < \infty
\end{equation}
and 
\begin{equation}\label{divinfsig}
\int ^\infty \bigg(\frac t{A(t)}\bigg)^{\frac 1{\sigma -1}}\,dt = \infty.
\end{equation}
Then, we define the function
$H_\sigma : [0, \infty) \to [0, \infty)$   as 
\begin{equation}\label{Hsig}
H_\sigma (s) = \bigg(\int _0^s \bigg(\frac t{A(t)}\bigg)^{\frac 1{\sigma-1}}\,dt \bigg)^{\frac 1{\sigma '}} \qquad \hbox{for $s \geq 0$,}
\end{equation}
and the Young function  $A_\sigma $ as
\begin{equation}\label{Bsig}
A_\sigma (t) = A(H_\sigma ^{-1}(t)) \qquad \hbox{for $t \geq 0$.}
\end{equation}

\begin{theorem}\label{sobolev}
Let $\Omega$ be an open set in $\rn$ with $|\Omega|<\infty$. Assume that the Young function $A$  and  the number $\sigma\geq n$ fulfill conditions \eqref{conv0sig} and \eqref{divinfsig}.
Then, there exists a  constant  $c=c(n, \sigma)$ such that
\begin{equation}\label{may95}
\int _\Omega A_\sigma \Bigg(\frac{|u(x)|}{c|\Omega|^{\frac 1n - \frac 1\sigma}\big(\int _\Omega  A(|\nabla  u |)dy\big)^{1/\sigma}}\Bigg)\, dx \leq  \int _\Omega A(|\nabla u|)dx
\end{equation}
for every $u \in W^{1,A}_0(\Omega)$.
\end{theorem}

\par\noindent
\begin{proof}
By the  P\'olya-Szeg\"o principle on the decrease of the functional on the right-hand side of inequality \eqref{may95} under symmetric decreasing rearrangement of functions $u \in W^{1,A}_0(\Omega)$ (see \cite{BZ}), it suffices to prove inequality \eqref{may95} in the case when $\Omega$ is a ball and the trial functions $u$ are nonnegative and radially decreasing.  As a consequence, this inequality will follow if we show that
\begin{equation}\label{may110}
\int _0^{|\Omega|} A_\sigma \Bigg(\frac{\int_s^{|\Omega|}\phi(r)r^{-\frac 1{n'}}\, dr}{c |\Omega|^{\frac 1n - \frac 1\sigma} \big(\int _0^{|\Omega|}  A(\phi(r))dr\big)^{1/\sigma}}\Bigg)\, ds \leq \int _0^{|\Omega|}  A(\phi(s))\, ds
\end{equation}
for a suitable constant $c$   as in the statement and for every measurable function $\phi : (0, |\Omega|)\to [0, \infty)$. Let $S$ be the linear operator defined as
\begin{equation}\label{S} 
S\phi(s) = \int_s^{|\Omega|}\phi(r)r^{-\frac 1{n'}}\, dr \quad \text{for $s \in (0, |\Omega|)$,}
\end{equation}
for every measurable function $\phi :  (0, |\Omega|) \to \mathbb R$ that makes the integral on the right-hand side converge. One has that
\begin{align}\label{may111}
\|S\phi\|_{L^{\sigma '}(0, |\Omega|)} & = \bigg(\int_0^{|\Omega|}|S\phi(s)|^{\sigma '}\, ds\bigg)^{\frac 1{\sigma '}} \leq \bigg(\int_0^{|\Omega|}s^{-\frac{\sigma '}{n'}}\bigg(\int_s^{|\Omega|}|\phi(r)|\, dr\bigg)^{\sigma '}\, ds\bigg)^{\frac 1{\sigma '}}
\\ \nonumber &\leq \|\phi\|_{L^1(0, |\Omega|)} \bigg(\int_0^{|\Omega|}s^{-\frac{\sigma '}{n'}}\, ds\bigg)^{\frac 1{\sigma '}} =c |\Omega|^{\frac 1n - \frac 1\sigma} \|\phi\|_{L^1(0, |\Omega|)}
\end{align}
for a suitable constant $c=c(n, \sigma)$ and for every $\phi \in L^1(0, |\Omega|)$. Also, by the Hardy-Littlewood inequality for rearrangements,
\begin{align}\label{may111bis}
\|S\phi\|_{L^{\infty}(0, |\Omega|)} & \leq  \int_0^{|\Omega|}|\phi(r)|r^{-\frac 1{n'}}\, dr 
\leq  \int_0^{|\Omega|}\phi^*(r)r^{-\frac 1{n'}}\, dr \\ \nonumber & \leq |\Omega|^{\frac 1n - \frac 1\sigma} \int_0^{|\Omega|}\phi^*(r)r^{-\frac 1{\sigma'}}\, dr =  |\Omega|^{\frac 1n - \frac 1\sigma} \|\phi\|_{L^{\sigma, 1}(0, |\Omega|)}
\end{align}
for every $\phi \in L^{\sigma, 1}(0, |\Omega|)$.  Here, $\varphi^*$ denotes the decreasing rearrangement of $\varphi$, and  $L^{\sigma, 1}(0, |\Omega|)$ is the Lorentz space whose norm is defined by the last integral in equation \eqref{may111bis}. Owing to equations  \eqref{may111} and \eqref{may111bis}, the interpolation theorem established in  \cite[Theorem 4]{Ci_CPDE} can be applied to deduce inequality \eqref{may110}.
\end{proof}

The next lemma tells us that the assumptions of Theorem \ref{sobolev} are certainly fulfilled   if $A$  satisfies  the $\Delta_2$-condition, provided that $\sigma$ is sufficiently large.

\begin{lemma}\label{aux}
Let $A$ be a continuously differentiable Young function satisfying the $\Delta_2$-condition and let $\sigma >s_A$. Then conditions \eqref{conv0sig} and \eqref{divinfsig} are fulfilled.
\end{lemma}
\begin{proof}
Owing to the definition of $s_a$, one verifies via differentiation that the function $\frac{A(t)}{t^{s_A}}$ is non-increasing. Thus,
\begin{equation}\label{may106} A(t) \geq A(1) t^{s_A}  \quad \text{if $t\in (0,1]$,}
\end{equation}
and 
\begin{equation}\label{may107}  A(t) \leq A(1) t^{s_B}    \quad \text{if $t\in [1, \infty)$.}
\end{equation}
Equations \eqref{conv0sig} and \eqref{divinfsig} follow from \eqref{may106} and \eqref{may107}, respectively.
\end{proof}

\subsection{\texorpdfstring{Young functions built upon  the function $a$}{Young functions built upon  the function a}}
\label{YoungB}

 \par Given a continuously differentiable  function $a: (0, \infty) \to (0,\infty)$ such that $i_a \geq -1$, let $b$ and $B$ the functions defined by \eqref{b} and \eqref{B}.
%
%
Our assumption on $i_a$ ensures that $b$ is a non-decreasing function, and hence  $B$  is a Young function. 
\\
 One has that 
\begin{equation}\label{jan200}
\text{$i_b= i_a+1$\quad  and \quad $s_b=s_a+1$.}
\end{equation}
Also 
\begin{equation}\label{jan201}
\text{$i_B\geq i_b+1$ \quad and \quad  $s_B\leq s_b +1$.}
\end{equation}
Thus, 
if $s_a<\infty$, then the functions $b$ and $B$ satisfy the $\Delta _2$-conditon, and if $i_a>-1$, then the function $\widetilde B$ satisfies the $\Delta _2$-conditon. 
\\ 
Hence, if $s_a<\infty$, then for every   $\lambda >1$, there exists a constant $c=c(\lambda, s_a)>1$ such that 
\begin{equation}\label{bdelta2}
b(\lambda t) \leq c b(t) \qquad \hbox{for $t \geq 0$,}
\end{equation}
and 
\begin{equation}\label{Bdelta2}
B(\lambda t) \leq c B(t) \qquad \hbox{for $t \geq 0$.}
\end{equation}
Moreover,
\begin{align}\label{may100}
tb'(t) \leq (s_a + 1) b(t) \quad \text{for $t>0$,}
\end{align}
and
\begin{align}\label{may101}
B(t) \leq tb(t) \leq  (s_a + 2) B(t) \quad \text{for $t>0$.}
\end{align}
Since $\widetilde B (b(t)) \leq B(2t)$ for $t \geq 0$, 
there exists a constant $c=c(s_a)$ such that
\begin{equation}\label{may41}
\widetilde B (b(t)) \leq c B(t) \quad \text{for $t\geq 0$.}
\end{equation}
Finally, if $i_a>-1$ and $s_a<\infty$, then
\begin{equation}\label{dic99}
a(1)\min\{t^{i_a}, t^{s_a}\} \leq a(t) \leq a(1)\max\{t^{i_a}, t^{s_a}\} \qquad \text{for $t>0$.}
\end{equation}

If the function $a$ is as above and $\varepsilon >0$, we define the function $a_\varepsilon  : [0,\infty) \to (0,\infty)$ as
\begin{equation}\label{aeps}
a_\varepsilon (t) = a(\sqrt {t^2 + \varepsilon^2})\quad \hbox{for $t \geq 0$.}
\end{equation}
The functions $b_\ep$ and $B_\ep$ are defined as in \eqref{b} and \eqref{B}, with $a$ replaced by $a_\ep$.

\begin{lemma}\label{aeps1}
Assume that the function $a: (0, \infty) \to (0, \infty)$ is continuously differentiable in $(0, \infty)$ and that $i_a>-1$ and $s_a<\infty$. Let   $ \varepsilon>0$ and let $a_ \varepsilon$ be the function defined by \eqref{aeps}. Then
\begin{equation}\label{epsindex}
i_{a_\ep} \geq  \min \{i_a, 0\} \quad \hbox{and} \quad s_{a_\ep} \leq \max \{s_a, 0\}\,,
\end{equation}
where $i_{a_\ep}$ and $s_{a_\ep}$ are defined as in \eqref{ia}, with $a$ replaced by $a_\ep$.
\\ Let $b$, $B$, $b_\ep$ and  $B_\ep$ be the functions defined above. 
Then there exist constants $c_1, c_2, c_3$, depending only on $s_a$, such that
\begin{equation}\label{may40}
c_1B(t) -c_2 B(\ep) \leq a_\ep(t) t^2 \leq c_3(B(t) + B(\ep)) \quad \text{for $t \geq 0$.}
\end{equation}
Moreover, there exists a constant $c=c(s_a)$ such that
\begin{equation}\label{nov12}
 B_\ep (t)   \leq c(B(t) + B(\ep)) \quad \text{for $t \geq 0$,}
\end{equation}
and
\begin{equation}\label{may60}
\widetilde B (b_\ep (t))   \leq c(B(t) + B(\ep)) \quad \text{for $t \geq 0$.}
\end{equation}
\end{lemma}

\begin{proof}
Property \eqref{epsindex} can be verified by straightforward computations. Consider equation \eqref{may40}.  One has  that
\begin{align}\label{nov1}
a_\ep(t) t^2 \leq a(t+\ep) t^2 \leq (s_a+2)B(t+ \ep)\leq (s_a+2) (B(2t) + B(2\ep)) \leq c(B(t)+B(\ep)) \quad \text{for $t\geq 0$,}
\end{align}
for some constant $c=c(s_a)$,
where the second inequality holds by \eqref{may101} and the last one by \eqref{delta2}. This proves the second inequality in \eqref{may40}. As for the first one, observe that
\begin{align}\label{nov2}
B(t) \leq B(t+\ep) \leq B(2t) + B(2\ep) \leq c B(t) + c B(\ep)  \quad \text{for $t\geq 0$,}
\end{align}
for some constant $c=c(s_a)$, where we have made use of inequality \eqref{delta2} again. Now, 
\begin{align}\label{nov3}
B(t) & = \int_0^t a(\tau)\tau \, d\tau \leq  \int_0^t a(\tau +\ep)(\tau +\ep) \, d\tau \leq  \int_0^t a(2\sqrt{\tau^2 +\ep^2})2\sqrt{\tau^2 +\ep^2} \, d\tau \\ 
\nonumber & \leq c  \int_0^t a(\sqrt{\tau^2 +\ep^2})\sqrt{\tau^2 +\ep^2} \, d\tau  \leq c\, t \,a(\sqrt{t^2 +\ep^2})\sqrt{t^2 +\ep^2}= c \,a_\ep (t) t \sqrt{t^2 +\ep^2} \quad \text{for $t\geq 0$,}
\end{align}
for some constant $c=c(s_a)$, where the third inequality is due to \eqref{bdelta2}. On the other hand,
\begin{align}\label{nov4} 
a_\ep (t) t \sqrt{t^2 +\ep^2} \leq \sqrt 2 a_\ep (t) t^2 \qquad \text{if $t\geq \ep$,}
\end{align}
and 
\begin{align}\label{nov5} 
a_\ep (t) t \sqrt{t^2 +\ep^2} \leq \sqrt 2 a_\ep (\ep) \ep^2  = \sqrt 2 a (\sqrt 2 \ep) \ep^2 \leq   c   B(\ep)\qquad \text{if $0\leq t\leq \ep$,}
\end{align}
for some constant $c=c(s_a)$, where the last  inequality holds thanks to \eqref{may101}. Combining inequalities \eqref{nov3}--\eqref{nov5} yields
\begin{align*}
B(t) \leq c a_\ep (t) t^2  + c B(\ep)  \qquad \text{for $t\geq 0$,}
\end{align*}
for some constant $c=c(s_a)$. Hence, the first inequality in \eqref{may40} follows.
\\ Inequality \eqref{nov12} holds because of the first inequality in \eqref{may101}, applied with $B$ replaced by $B_\ep$, and of the second inequality in \eqref{may40}. 
\\ Inequality
\eqref{may60} is a consequence of the following chain:
\begin{align}\label{nov6}
\widetilde B(b_\ep (t) ) & = \widetilde B(a (\sqrt{t^2 +\ep^2}) t)\leq  \widetilde B(b (\sqrt{t^2 +\ep^2}))\\ \nonumber & \leq 
 \widetilde B(b (t+\ep))  \leq c B(t+\ep) \leq c' (B(t) + B(\ep))  \quad \text{for $t\geq 0$,}
\end{align}
for some constants $c$ and $c'$ depending on $s_a$. Notice, that
 we have made use of property \eqref{may41} in last but one inequality, and of property \eqref{delta2} in the last inequality.
\end{proof}

\begin{lemma}\label{holder}
Assume that the function $a: (0, \infty) \to (0, \infty)$ is continuously differentiable in $(0, \infty)$ and that $i_a>-1$ and $s_a<\infty$. Let   $ \varepsilon>0$ and let $a_ \varepsilon$ be the function defined by \eqref{aeps}. Let $M>0$. Then there exists a constant $c=c(i_a, s_a, \ep, M)$ such that
\begin{equation}\label{holder1}
|P-Q| \leq c |a_\ep(P)P -a_\ep(Q)Q|
\end{equation}
for every $P, Q \in \mathbb R^{N\times n}$ such that $|P|\leq M$ and $|Q|\leq M$.
\end{lemma}
\begin{proof}
By \cite[Lemma 21]{DE}, there exists a positive constant $c=c(i_{a_\ep}, s_{a_\ep})$ such that
\begin{equation}\label{holder2}
c \big[a_\ep (|P|+|Q|) + a_\ep'(|P|+|Q|) (|P|+|Q|)\big] |P-Q|^2 \leq (a_\ep(|P|)P - a_\ep(|Q|)Q) \cdot (P-Q) 
\end{equation}
for every $P, Q \in \mathbb R^{N\times n}$. Hence, via inequalities \eqref{epsindex},
\begin{equation}\label{holder3}
c (1+ \min\{i_a, 0\}) a_\ep (|P|+|Q|)  |P-Q| \leq |a_\ep(|P|)P - a_\ep(|Q|)Q|
\end{equation}
for every $P, Q \in \mathbb R^{N\times n}$. Inequality \eqref{holder} hence follows, since
$$a_\ep (|P|+|Q|)  \geq \min \big\{a(t): \ep \leq t \leq \sqrt{2M^2+ \ep^2}\big\}>0$$
if $|P|\leq M$ and $|Q|\leq M$, and  $(1+ \min\{i_a, 0\})>0$.
\end{proof}
One more function associated with a function $a$ as above and to a number $\ep >0$ will be needed in our proofs. The function in question is denoted by 
$V_\ep : \rNn \to \rNn$ and is defined as 
\begin{equation}\label{nov8}
V_\ep (P)= \sqrt {a_\ep(|P|)}P \qquad \text{ for $P \in \rNn$.}
\end{equation}

\begin{lemma}\label{aeps2} Assume that the function $a: (0, \infty) \to (0, \infty)$ is continuously differentiable  and such that  $i_a>-1$ and $s_a<\infty$.
Let   $ \varepsilon>0$ and let $a_ \varepsilon$ be the function defined by \eqref{aeps}. 
Then 
\begin{equation}\label{may51}
a_\ep(|P|) P \to a(|P|) P \quad \text{as $\ep \to 0^+$,}
\end{equation}
uniformly for $P$ in any compact subset of $\rNn$.
\\ Moreover,  
\begin{equation}\label{may70}
(a_\ep(|P|) P - a_\ep(|Q|) Q)\cdot (P - Q) \approx \big|V_\ep (P) -V_\ep (Q)\big|^2  \quad \text{for $P, Q \in \rNn$,}
\end{equation}
where the relation $\approx$ means that the two sides are bounded by each other, up to positive multiplicative constants depending only on $i_a$ and $s_a$.
\end{lemma}

\begin{proof}
Fix any $0<\ell < L$ and assume that $\ep \in [0, 1]$. Since $a\in C^1(0, \infty)$, if $\ell \leq |P| \leq L$ then  
\begin{align}\label{dic100}
|a_\ep(|P|)P - a(|P|)P|  &\leq |P| |a_\ep(|P|) - a(|P|)| \\ \nonumber&\leq \max_{t\in [\ell, \sqrt{L^2+1}]}|a'(t)|(\sqrt{|P|^2+\ep^2}- |P|)\leq \max_{t\in [\ell, \sqrt{L^2+1}]}|a'(t)| \ep.
\end{align}
Moreover, if $|P|\leq 1$, then, by the second inequality in \eqref{dic99} applied with $a$ replaced by $a_\ep$ and by the first inequality in \eqref{epsindex},
\begin{equation}\label{dic101}
|a_\ep(|P|)P| \leq a_\ep (1) |P|^{1+ \min\{i_a,0\}} \leq \max_{t\in [1, \sqrt{2}]}|a(t)| |P|^{1+ \min\{i_a,0\}}.
\end{equation}
Now, let $L>0$. Fix any $\sigma >0$. By inequality \eqref{dic101}, there exists $\ell>0$ such that
\begin{equation}\label{dic102}
|a_\ep(|P|)P-a(|P|)P| \leq |a_\ep(|P|)P|+ |a(|P|)P| \leq \sigma
\end{equation}
for every $\ep \in [0, 1]$, provided that $|P|<\ell$. On the other hand, inequality \eqref{dic100} ensures that there exists $\ep_0\in (0,1)$ such that
\begin{equation}\label{dic103}
|a_\ep(|P|)P - a(|P|)P|<\sigma
\end{equation}
if $\ell \leq |P| \leq L$. From inequalities \eqref{dic102} and \eqref{dic103} we deduce that, if $0\leq \ep <\ep_0$, then 
\begin{equation}\label{dic104}
|a_\ep(|P|)P - a(|P|)P|<\sigma \qquad \text{if $|P|\leq L$.}
\end{equation}
This shows that the limit \eqref{may51} holds unifromly for $|P|\leq L$.
\\ As far as equation \eqref{may70} is concerned, it follows from \cite[Lemma 41]{DFTW} that,  if $i_{a_\ep}>-1$ and $s_{a_\ep}<\infty$, then the ratio of the two sides of this equation is bounded from below and from above by positive constants depending only on a lower bound for $i_{a_\ep}$ and an upper bound for $s_{a_\ep}$. Owing to inequalities \eqref{epsindex} and to our assumption that $i_a>-1$ and $s_a<\infty$, we have that $i_{a_\ep}\geq \min\{i_a, 0\}>0$ and $s_{a_\ep}\leq \max\{s_a, 0\}<\infty$ for every $\ep>0$. This implies that equation \eqref{may70} actually holds up to equivalence constants depending only on $i_a$ and $s_a$.
%
%
%
\end{proof}

\section{Second-order regularity:   local solutions}\label{S:local}

 The definiton of generalized local solution to the system
\begin{equation}\label{localeqbis}
- {\rm {\bf  div}}( a(|\nabla \bu|) \nabla {\bf u} ) = {\bf f}  \quad {\rm in}\,\,\, \Omega\,
\end{equation}
that will be adopted is inspired by the results of \cite{DHM}, and involves  the notion of approximate differentiability. Recall that a measurable function $\bu  : \Omega \to \rN$ is said to be approximately differentiable at $x \in \Omega$ if there exists a matrix ${\rm ap} \nabla \bu(x)  \in \mathbb R^{N\times n}$ such that, for every $\varepsilon >0$,
$$\lim_{r \to 0^+} \frac {\big|\{y\in B_r(x): \frac  1r |\bu(y)-\bu(x)- {\rm ap} \nabla \bu(x)  (y-x)|>\varepsilon\}\big|}{r^n} =0.$$
If $\bu$ is approximately differentiable at every point in $\Omega$, then the function ${\rm ap} \nabla \bu : \Omega \to \mathbb R^{N\times n}$ is measurable.
\par Assume that $a$ is as in Theorem \ref{secondloc} and let $\bff \in L^q_{\rm loc}(\Omega, \rN)$ for some $q\geq 1$. 
An approximately differentiable function $\bu : \Omega \to \rN$ is called a local  approximable  solution to system
\eqref{localeqbis}
if $a(|{\rm ap}\nabla \bu|) |{\rm ap}\nabla \bu| \in L^{1}_{\rm loc}(\Omega)$, 
and
there exist a sequence $\{\bff_k\}\subset  C^\infty (\Omega, \rN)$, with
$\bff_k \to \bff$   in $L^q_{\rm loc} (\Omega, \rN)$,
and a corresponding sequence of local weak solutions $\{\bu_k\}$ to the systems
\begin{equation}\label{localeqk}
- {\rm {\bf  div}}( a(|\nabla \bu_k|) \nabla {\bf u} _k) = {\bf f}_k  \quad {\rm in}\,\,\, \Omega\,,
\end{equation}
such that
\begin{equation}\label{approxuloc}
\bu_k \to \bu \quad \hbox{and} \quad \nabla \bu_k \to {\rm ap} \nabla \bu \quad \hbox{a.e. in $\Omega$,}
\end{equation}
and
\begin{equation}\label{approxaloc}
\lim _{k \to \infty} \int_{\Omega'}a(|\nabla \bu_k|) |\nabla \bu_k| \, dx =\int_{\Omega'} a(|{\rm ap}\nabla \bu|) |{\rm ap}\nabla \bu|\, dx\,
\end{equation}
for every open set $\Omega' \subset \subset \Omega$. In what follows, we shall denote ${\rm ap}\nabla \bu$ simply by $\nabla \bfu$.

Weak solutions to  system \eqref{localeqbis}
are defined in a standard way
if  $\bff \in L^1_{\rm loc}(\Omega, \rN) \cap (W^{1,B}_0(\Omega, \rN))'$, where $B$ is the Young function defined via  \eqref{B}. Namely, a function $\bu \in W^{1,B}_{\rm loc}(\Omega, \mathbb R^N)$ is called a local weak solution to this system 
 if
\begin{equation}\label{weaklocl}
\int _{\Omega'} a(|\nabla \bu|)\nabla \bu \cdot \nabla  \bfvarphi \color{black} \,dx = \int _{\Omega '} {\bff}\cdot \bfvarphi  \,dx
\end{equation}
for every open set $\Omega ' \subset \subset \Omega$, and every function $\bfvarphi \in W^{1,B}_0(\Omega', \mathbb R^N)$.

\smallskip
\par
Inequality \eqref{pointwise} enters the proof of Theorem  \ref{secondloc} through Lemma  \ref{step1local} below. The latter will be applied to solutions to systems  which approximate system \eqref{localeq}, and involve regularized differential operators and smooth right-hand sides. Lemma \ref{step1local}
 can be deduced from Theorem \ref{lemma1} and inequality \eqref{>0}, along the same lines as in the proof of \cite[Theorem 3.1, Inequality (3.4)]{CiMa_JMPA}.  The details are omitted, for brevity.  We seize this opportunity to point out an incorrect dependence on the radius $R$ of the constants  in that inequality, due to a flaw in the scaling argument in the derivation of \cite[Inequality (3.43)]{CiMa_JMPA}.

\begin{lemma}\label{step1local}
Let $n \geq 2$, $N \geq 2$, and let $\Omega$ be an open set in $\rn$. Assume that the function $a \in C^1([0, \infty))$ satisfies conditions \eqref{positive}--\eqref{bC1}.
%
%
Then there exists a constant $C= C(n,N, i_a, s_a)$,  such that
\begin{multline}\label{loc15}
R^{-1}\bignorm{a(|\nabla \bu|) \nabla \bu}_{L^2(B_R, \rNn)} + \,\bignorm{\nabla \big(a(|\nabla \bu|) \nabla \bu\big)}_{L^2(B_R, \rNn)}\\ \leq C\Big(\|{\rm {\bf div}} ( a(|\nabla \bu|) \nabla \bu)\|_{L^2(B_{2R}, \mathbb R^{N})} + R^{-\frac n2-1}\|a(|\nabla \bu|)\nabla \bu\|_{L^1(B_{2R}, \mathbb R^{N\times n})}\Big)
\end{multline}
for every function $\bu \in C^3(\Omega, \rN)$ and any ball $B_{2R}  \subset \subset \Omega$.
\end{lemma}
%
%
%
%
\begin{proof}[Proof of Theorem \ref{secondloc}]
Let us temporarily assume that
\begin{equation}\label{fsmoothloc}
\bff \in C^\infty (\Omega, \rN)\,,
\end{equation}
and that $\bfu$ is a local weak solution to system \eqref{localeqbis}.
Observe that, thanks to equations \eqref{inf} and \eqref{epsindex},
\begin{equation}\label{epsindex1}
i_{a_\ep} > 2(1-\sqrt 2)\,.
\end{equation}
Let  $B_{2R} \subset \subset \Omega$ 
and, given $\ep \in (0,1)$,  let  $\bu_\varepsilon \in \bfu + W^{1,B}_0(B_{2R}, \rN)$ be the weak solution to the Dirichlet problem 
\begin{equation}\label{eqeploc}
\begin{cases}
- {\rm{\bf div}} (a_\ep(|\nabla \bu_\ep|)\nabla \bu_\ep ) = \bff  & {\rm in}\,\,\, B_{2R} \\
 \bu_\ep =\bu  &
{\rm on}\,\,\,
\partial B_{2R} \,.
\end{cases}
\end{equation}
We claim that 
\begin{align}\label{cinf}
\bu _\ep \in C^{\infty}(B_{2R}, \rN).
\end{align} 
Actually,    as a consequence of  \cite[Corollary 5.5]{DKS},    $\nabla \bu_\ep \in L^{\infty}_{\rm loc}(B_{2R}, \rNn)$ and there exists a constant $C$, independent of $\ep$, such  that
\begin{equation}\label{dec210}
\|\nabla u_\ep\|_{L^\infty(B_{R}, \rNn)}\leq C.
\end{equation}
The same result also tells us that
 $a_\ep(|\nabla \bu_\ep|)\nabla \bu_\ep \in  C^{\alpha}_{\rm loc}(B_{2R}, \rNn)$ for some  $\alpha \in (0,1)$.  Therefore, by inequality \eqref{holder1}, we have that $\nabla \bu_\ep \in C^{\alpha}_{\rm loc}(B_{2R}, \rNn)$  as well.
Hence,  
   $a_\ep (|\nabla \bu _\ep|) \in C^{1,\alpha}_{\rm loc}(B_{2R})$, and by  the  Schauder theory for linear elliptic systems,  $\bu _\ep \in C^{2,\alpha}_{\rm loc}(B_{2R}, \rN)$.  An iteration argument relying upon the the Schauder theory again yields property \eqref{cinf}.
\\
We claim that
\begin{align}\label{cialpha1}
\int _{B_{2R}} B(|\nabla \bu _\ep|)\, dx \leq C \bigg(\int _{B_{2R}}\widetilde B(|\bff|)\,dx +  \int _{B_{2R}}B(|\nabla \bu |)\, dx + B(\ep)\bigg)\,
\end{align}
for some  constant $C=C(n, N, s_a, R)$ and
for $\ep \in (0,1)$.  
Indeed, choosing   $\bu _\ep - \bu \in W^{1,B}_0(B_{2R}, \rN)$ as a test function in the weak formulation of problem \eqref{eqeploc}  results in
\begin{align}\label{cialpha2}
\int _{B_{2R}} a_\ep (|\nabla \bu _\ep|)\nabla \bu _\ep \cdot (\nabla \bu _\ep - \nabla \bu)\, dx = \int _{B_{2R}} \bff \cdot (\bu _\ep - \bu)\, dx\,.
\end{align}
The Poincar\'e   inequality  \eqref{poinc} implies that 
\begin{equation}\label{poincB}
\int_{B_{2R}}B(|\bu _\ep - \bu|)\, dx \leq C \int_{B_{2R}}B(|\nabla \bu _\ep - \nabla \bu|)\, dx
\end{equation}
for some constant $C=C(n, s_a, R)$.
\\ Fix $\delta \in (0,1)$. From equation \eqref{cialpha2}, the first inequality in \eqref{may40},  and   inequalities \eqref{poincB} ,    \eqref{may41} and \eqref{nov12} one obtains that
\begin{align}\label{cialpha3}
c_1\int _{B_{2R}}B(|\nabla \bu _\ep|)\, dx  \leq  & \int _{B_{2R}} |\bff|  |\bu _\ep - \bu|\, dx + C \int _{B_{2R}}a_\ep(|\nabla \bu _\ep|) |\nabla \bu _\ep||\nabla \bu|\, dx   + C R^n B(\ep)
\\ \nonumber & 
\leq    C_1 \int _{B_{2R}} \widetilde B(|\bff|)\, dx  +  \delta \int _{B_{2R}}  B(|\bu _\ep - \bu|)\, dx 
\\ \nonumber & \quad + \delta \int _{B_{2R}}\widetilde B_\ep(a_\ep(|\nabla \bu _\ep|) |\nabla \bu _\ep|)\, dx + C_1  \int _{B_{2R}}B_\ep(|\nabla \bu|)\, dx + C R^n B(\ep)
\\ \nonumber & 
\leq     C_1 \int _{B_{2R}} \widetilde B(|\bff|)\, dx  +    \delta C_2 \int _{B_{2R}}B(|\nabla \bu _\ep|)\, dx + C_3 \int _{B_{2R}}B(|\nabla \bu|)\, dx 
\\ \nonumber & \quad +  \delta C_4 \int _{B_{2R}}B_\ep(|\nabla \bu _\ep|)\, dx + C_1 \int _{B_{2R}}B_\ep(|\nabla \bu |)\, dx + C R^n  B(\ep)
\\ \nonumber
& \leq     C_1 \int _{B_{2R}} \widetilde B(|\bff|)\, dx  +    \delta C_5 \int _{B_{2R}}B(|\nabla \bu _\ep|)\, dx + C_6  \int _{B_{2R}}B(|\nabla \bu|)\, dx + C R^n B(\ep)
\end{align}
for suitable constants  $C_2$, $C_4$ and $C_5$ depending  on 
$n, N, s_a, R$, and constants $C_1$, $C_3$ and $C_6$ depending also on $\delta$.
Inequality   \eqref{cialpha1} follows from  \eqref{cialpha3}, on  choosing $\delta$ small enough.
\\ Coupling inequality \eqref{cialpha1} with the Poincar\'e inequality \eqref{poinc} tells us that the family $\{\bfu_\ep\}$ is bounded in $W^{1,B}(B_{2R}, \rN)$. Since under assumptions \eqref{inf} and \eqref{sup} this space is reflexive, there exist a sequence $\{\ep_k\}$ and a function $\bfv \in W^{1,B}(B_{2R}, \rN)$ such  that $\ep_k \to 0^+$ and 
\begin{equation}\label{may43}
\bfu_{\ep_k} \rightharpoonup \bfv \qquad \text{in $W^{1,B}(B_{2R}, \rN)$.}
\end{equation}
 Choosing  the test function $\bfu_{\ep_k} - \bfu$ for system \eqref{localeq}, and subtracting the resultant equation from \eqref{cialpha2} enables us to deduce that, given any $\delta >0$,
\begin{align}\label{akm1}
\int_{B_{2R}}&\big(a_{\ep_k}(|\nabla \bu _{\ep_k}|)\nabla \bu_{\ep_k} - a_{\ep_k}(|\nabla \bu|)\nabla \bu\big)\cdot (\nabla \bu _{\ep_k} - \nabla \bu)\, dx \\ \nonumber & = 
\int_{B_{2R}}\big(a(|\nabla \bu|)\nabla \bu - a_{\ep_k}(|\nabla \bu|)\nabla \bu\big) \cdot(\nabla \bu _{\ep_k} - \nabla \bu)\, dx 
\\ \nonumber & \leq \delta \int_{B_{2R}} B(|\nabla \bu_{\ep_k}|) + B(|\nabla \bu|) \, dx + C
 \int_{B_{2R}}  \widetilde B\big(|a(|\nabla \bu|)\nabla \bu - a_{\ep_k}(|\nabla \bu|) \nabla \bu|\big) dx 
\end{align}
for some constant $C=C(\delta, s_a)$. Owing to equation \eqref{may70}, there exists a constant $c=c(i_a, s_a)$ such that
\begin{align}\label{may71}
\int_{B_{2R}}& |V_{\ep_k} (\nabla \bfu_{\ep_k}) - V(\nabla \bfu )|^2\,dx    \leq 
2\int_{B_{2R}} |V_{\ep_k} (\nabla \bfu_{\ep_k}) - V_{\ep_k}(\nabla \bfu )|^2\,dx +2\int_{B_{2R}} |V_{\ep_k} (\nabla \bfu) - V(\nabla \bfu )|^2\,dx 
\\ \nonumber & \leq c 
\int_{B_{2R}}\big(a_{\ep_k}(|\nabla \bu _{\ep_k}|)\nabla \bu_{\ep_k} - a_{\ep_k}(|\nabla \bu|)\nabla \bu\big)\cdot (\nabla \bu _{\ep_k} - \nabla \bu)\, dx +
 2\int_{B_{2R}} |V_{\ep_k} (\nabla \bfu) - V(\nabla \bfu )|^2\,dx.
\end{align}
Combining equations \eqref{may71}, \eqref{akm1} and \eqref{cialpha1}  yields
\begin{align}\label{may72}
\int_{B_{2R}} |V_{\ep_k} (\nabla \bfu_{\ep_k}) -  V(\nabla \bfu )|^2\,dx    & \leq 
\delta c \bigg(\int _{B_{2R}}\widetilde B(|\bff|)\,dx +  \int _{B_{2R}}B(|\nabla \bu |)\, dx + B(\ep)\bigg)  \\ \nonumber & \quad +c
 \int_{B_{2R}}  \widetilde B\big(|a(|\nabla \bu|)\nabla \bu - a_{\ep_k}(|\nabla \bu|) \nabla \bu|\big) dx 
 \\ \nonumber &
\quad + 2\int_{B_{2R}} |V_{\ep_k} (\nabla \bfu) - V(\nabla \bfu )|^2\,dx
\end{align}
for some constant $c=c(n,N,R,i_a,s_a)$.
 Inequalities \eqref{may41} and \eqref{may60} entail that
\begin{align}\label{may73} 
\widetilde B\big(|a(|\nabla \bu|)\nabla \bu - a_{\ep_k}(|\nabla \bu|) \nabla \bu|\big) \leq c (B(|\nabla \bfu|) + B({\ep_k})) 
\quad \text{a.e. in $B_{2R}$, }
\end{align}
for some constant $c=c(s_a)$.
Furthermore, from inequality \eqref{may40} one infers that
\begin{align}\label{may74} 
|V_{\ep_k} (\nabla \bfu)|^2 \leq c (B(|\nabla \bfu|) + B({\ep_k}))  \quad \text{a.e. in $B_{2R}$,}
\end{align}
for some constant $c=c(s_a)$.
Thanks to inequalities \eqref{may73} and \eqref{may74}, and  to property \eqref{may51}, the last two integrals on the right-hand side of inequality \eqref{may72} tend to $0$ as $k\to \infty$, via the dominated convergence theorem. Owing to the same theorem,   equation \eqref{may72} implies that
\begin{equation}\label{may75}
\lim _{k\to \infty}\int_{B_{2R}} |V_{\ep_k} (\nabla \bfu_{\ep_k}) - V(\nabla \bfu )|^2\,dx \leq \delta c
\end{equation}
for every $\delta \in (0,1)$. Thereby, 
\begin{equation}\label{may45}
V_{\ep_k}(\nabla \bfu_{\ep_k}) \to V(\nabla \bfu ) \quad \text{in $L^{2}(B_{2R}, \rNn)$,}
\end{equation}
and, on passing to a subsequence, still indexed by $k$, 
\begin{equation}\label{may47}
V_{\ep_k}(\nabla \bfu_{\ep_k})  \to V(\nabla \bfu ) \quad \text{a.e. in $B_{2R}$.}
\end{equation}
An analogous argument as in \cite[Lemma 4.8]{DSV} shows that the function $(\ep, P) \mapsto V_{\ep}^{-1}(P)$ is continuous. Thus, one can deduce 
from equation \eqref{may47} that
\begin{equation}\label{may48}
\nabla \bfu_{\ep_k}  \to \nabla \bfu   \quad \text{a.e. in $B_{2R}$.}
\end{equation}
Hence, equation \eqref{may43} implies  that $\bfv=\bfu$ and
\begin{equation}\label{may76}
\bfu_{\ep_k} \rightharpoonup \bfu \qquad \text{in $W^{1,B}(B_{2R}, \rN)$.}
\end{equation}
Inequalities \eqref{may40} and \eqref{cialpha1}, and the monotonicity of the function $b_{\ep_k}$, yield
\begin{align}\label{may78}
\int_{B_{2R}}  a_{\ep_k}(|\nabla \bfu_{\ep_k}|) |\nabla \bfu_{\ep_k}|\, dx 
& \leq
\int_{\{|\nabla \bfu_{\ep_k}|\leq 1\} \cap B_{2R}}  a_{\ep_k}(|\nabla \bfu_{\ep_k}|) |\nabla \bfu_{\ep_k}|\, dx + \int_{B_{2R}}  a_{\ep_k}(|\nabla \bfu_{\ep_k}|) |\nabla \bfu_{\ep_k}|^2\, dx
\\ \nonumber & \leq cR^n b_{\ep_k}(1) +
c \int_{B_{2R}}B (\nabla \bfu_{\ep_k}) \,dx  + c R^n B(\ep_k) \leq C
\end{align}
for some constants $c$ and $C$ independent of $k$.
Thanks to assumption \eqref{epsindex1} and to property \eqref{epsindex},  Lemma \ref{step1local}  can be applied with $a$ replaced by $a_{\ep_k}$.
%
The use  of inequality \eqref{loc15} of this lemma for the function $\bfu_{\ep_k}$,    and the equation in \eqref{eqeploc}, ensure that
\begin{align}\label{loc18}
\|a_\ep(|\nabla \bu_{\ep_k}|)\nabla \bu_{\ep_k}&  \|_{W^{1,2}(B_R, \rNn)}  \\ \nonumber  & \leq C\big(\|\bff\|_{L^2(B_{2R}, \rN)} + (R^{-\frac n2}+R^{-\frac n2-1})\|a_{\ep_k}(|\nabla \bu_{\ep_k}|)\nabla \bu_{\ep_k}\|_{L^1(B_{2R}, \rNn)}\big)\,,
\end{align}
for some constant  $C=C(n,N, i_a, s_a)$. 
 Owing to inequalities \eqref{may78} and \eqref{loc18}, the sequence $\{a_{\ep_k}(|\nabla \bu_{\ep_k}|)\nabla \bu_{\ep_k}\}$ is bounded in $W^{1,2}(B_R, \rNn)$.  Thus, there exists a function $\bfU\in W^{1,2}(B_R, \rNn)$, and a subsequence of $\{\ep _k\}$, still indexed by $k$, such that 
\begin{align}\label{loc19}
a_{\ep _k}(|\nabla \bu_{\ep _k}|) \nabla \bu_{\ep _k} &\to \bfU\quad \hbox{in $L^2(B_R, \rNn)$} \quad
\\ 
\nonumber&\hbox{and} \quad a_{\ep _k}(|\nabla \bu_{\ep _k}|) \nabla \bu_{\ep _k} \rightharpoonup \bfU\quad \hbox{in $W^{1,2}(B_R, \rNn)$.}
\end{align}
Combining property \eqref{may51} with equations \eqref{dec210},
 \eqref{may48} and  \eqref{loc19}   yields
\begin{equation}\label{loc22}
a(|\nabla \bfu|) \nabla \bfu = \bfU \in W^{1,2}(B_R, \rNn).
\end{equation}
On passing to the limit as $k \to \infty$,
from equations \eqref{loc18}, \eqref{loc19}  and \eqref{loc22} we infer that
\begin{align}\label{loc26}
&\| a(|\nabla \bu|)\nabla \bu\|_{W^{1,2}(B_R, \rNn)}   \leq C\big(\|\bff\|_{L^2(B_{2R}, \rN)} + (R^{-\frac n2}+R^{-\frac n2-1})\|a(|\nabla \bu|)\nabla \bu\|_{L^{1}(B_{2R}, \rNn)}\big)\,.
\end{align}
It remains to remove assumption \eqref{fsmoothloc}. Suppose that
 $\bff\in L^2_{\rm loc}(\Omega, \rN)$. Let $\bu$ be an approximable local solution  to equation \eqref{localeq}, and let $\bff_k$ and $\bu_k$ be as in the  definition  of this kind of solution.
Applying inequality \eqref{loc26}  to the function $\bu_k$ tells us that 
$a(|\nabla \bu_k|) \nabla \bu_k   \in W^{1,2}(B_R, \rNn)$, and 
\begin{align}\label{loc27}
\|a(|\nabla \bu_k|)\nabla \bu_k&\|_{W^{1,2}(B_R, \rNn)}\\\nonumber&  \leq C\big(\|\bff_k\|_{L^2(B_{2R}, \rN)} + (R^{-\frac n2}+ R^{-\frac n2-1})\|a(|\nabla \bu_k|)\nabla \bu_k\|_{L^{1}(B_{2R}, \rNn)}\big)
\,,
\end{align}
for some constant $C$  independent of $k$. Hence, by equation \eqref{approxaloc}, the sequence $\{a(|\nabla \bu_k|)\nabla \bu_k\}$ is bounded in $W^{1,2}(B_R, \rNn)$. Thereby,
there exist  a subsequence, still indexed by $k$,   and a function
$\bfU \in W^{1,2}(B_R, \rNn)$, such that
\begin{equation}\label{loc28}
a(|\nabla \bu_k|)\nabla \bu_k \to \bfU\quad \hbox{in $L^2(B_R, \rNn)$} \quad \hbox{and} \quad  a(|\nabla \bu_k|)\nabla \bu_k \rightharpoonup \bfU\quad \hbox{in $W^{1,2}(B_R, \rNn)$}.
\end{equation}
By assumption \eqref{approxuloc}, we have that $\nabla \bu_k \to \nabla \bu$ a.e. in $\Omega$. Hence, thanks to properties  \eqref{loc28}, 
\begin{equation}\label{loc30}
a(|\nabla \bu|)\nabla \bu = \bfU\in W^{1,2}(B_R, \rNn)\,.
\end{equation}
 Inequality \eqref{eq:secondloc2} follows on passing to the limit as $k\to \infty$ in   \eqref{loc27},
 via \eqref{approxaloc}, \eqref{loc28}  and \eqref{loc30}. 
\end{proof}

\section{Second-order regularity:   Dirichlet problems}\label{S:global}


Generalized solutions, in the approximable sense, 
 to the Dirichlet problem 
\begin{equation}\label{eqdirbis}
\begin{cases}
-{\rm {\bf  div}} (a(|\nabla \bu|)\nabla {\bfu} ) = {\bff}  & {\rm in}\,\,\, \Omega\\
 {\bfu} =0  &
{\rm on}\,\,\,
\partial \Omega \,,
\end{cases}
\end{equation}
are defined in analogy with the local solutions introduced in Section \ref{S:local}. 
\\ Assume that  $a$ is as in Theorems \ref{secondconvexteo} and \ref{seconddir} and let $\bff \in L^q(\Omega, \rN)$ for some $q \geq 1$. 
An approximately differentiable function $\bu : \Omega \to \rN$ is called an approximable solution to the Dirichlet problem \eqref{eqdirbis}  if   there exists a sequence $\{{\bff}_k\} \subset C^{\infty}_0(\Omega, \mathbb R^N)$ such that  ${\bff}_k \to \bff$ in $L^q(\Omega, \mathbb R^N)$, and 
the sequence $\{\bu _k\}$ of weak solutions to the Dirichlet problems 
\begin{equation}\label{eqdirichletk}
\begin{cases}
-{\rm {\bf  div}} (a(|\nabla \bu _k|)\nabla {\bfu} _k) = {\bff}_k  & {\rm in}\,\,\, \Omega \\
 {\bfu}_k =0  &
{\rm on}\,\,\,
\partial \Omega \,
\end{cases}
\end{equation}
satisfies
\begin{equation}\label{convdir}
\bu _k \to \bu \quad \hbox{and} \quad \nabla \bu _k \to {\rm ap}\nabla \bu \quad \hbox{a.e. in $\Omega$.}
\end{equation}
As above, in what follows $ {\rm ap}\nabla \bu$ will simply be denoted by $\nabla \bfu$.
  
Recall that, 
under the assumption that $\bff \in L^1(\Omega, \rN) \cap (W^{1,B}_0(\Omega, \rN))'$,   a function $\bu \in W^{1,B}_0(\Omega, \mathbb R^N)$ is called a weak solution to the Dirichlet problem \eqref{eqdirbis}
 if  
\begin{equation}\label{weaksol}
\int _\Omega a(|\nabla \bu|)\nabla \bu \cdot \nabla  \bfvarphi  \,dx = \int _{\Omega} {\bff}\cdot \bfvarphi \,dx
\end{equation}
for every $\bfvarphi  \in W^{1,B}_0(\Omega, \mathbb R^N)$. A unique weak solution to problem \eqref{eqdirbis} exists whenever  $|\Omega|<\infty$.
\\

Before accomplishing the proof of our global estimates, we recall the notions of capacity and of Marcinkiewicz spaces that enter conditions \eqref{capcond} and \eqref{smalln}, respectively,  in the statement of Theorem  \ref{seconddir}.

The capacity ${\rm cap}_{\Omega}(E)$  of a set $E  \subset {\Omega}$ relative to ${\Omega}$ is defined as
\begin{equation}\label{cap}
{\rm cap}_{\Omega}(E) = \inf \bigg\{\int _{\Omega} |\nabla v|^2\, dx : v\in C^{0,1}_0({\Omega}), v\geq 1 \,\hbox{on }\, E\bigg\}.
\end{equation}
Here, $C^{0,1}_0({\Omega})$ denotes the space of  Lipschitz continuous, compactly supported functions in ${\Omega}$.

The Marcinkiewicz space $L^{q, \infty} (\partial \Omega)$ is the Banach function space endowed with the norm defined as 
\begin{equation}\label{weakleb}
\|\psi\|_{L^{q, \infty} (\partial \Omega)} = \sup _{s \in (0, \hh(\partial \Omega))} s ^{\frac 1q} \psi^{**}(s)
\end{equation}
for a measurable function $\psi$ on $\partial \Omega$. Here, $\psi ^{**}(s)= \frac 1s\smallint _0^s \psi^* (r)\, dr$ for $s >0$, where $\psi^*$ denotes the decreasing rearrangement of $\psi$.
%
The Marcinkiewicz   space $L^{1, \infty} \log L (\partial \Omega)$ is equipped with the norm given by
\begin{equation}\label{weaklog}
\|\psi\|_{L^{1, \infty} \log L (\partial \Omega)} = \sup _{s \in (0, \hh(\partial \Omega))} s \log\big(1+ \tfrac{C}s\big) \psi^{**}(s),
\end{equation}
for any constant  $C>\hh(\partial \Omega)$.  Different constants $C$ result in equivalent norms in \eqref{weaklog}.

\smallskip
\par
The next lemma stands with respect to  Theorems \ref{secondconvexteo} and \ref{seconddir} that Lemma \ref{step1local} stands to Theorem \ref{secondloc}. 
It follows from Theorem \ref{lemma1} and inequality \eqref{>0}, via the same  proof of \cite[Theorem 3.1, Part (ii)]{CiMa_ARMA}.

\begin{lemma}\label{stepglobal}
Let $n \geq 2$, $N \geq 2$, and let $\Omega$ be a bounded open set in $\rn$ with $\partial \Omega \in C^2$. Assume that $a$ is a function as in Theorem \ref{lemma1}, which also
%
fulfills conditions \eqref{inf} and \eqref{sup}.
There exists a constant $c=c(n,N, i_a, s_a, L_\Omega, d_\Omega)$ such that, if
\begin{equation}\label{KK}
\mathcal K_\Omega (r) \leq \mathcal  K(r) \quad \hbox{for $r \in (0,1)$,}
\end{equation}
for some function $\mathcal K : (0,1) \to [0, \infty)$ satisfying
\begin{equation}\label{limK}
\lim _{r \to 0^+} \mathcal  K(r) <c\,,
\end{equation}
then
%
\begin{align}\label{fund}
\|a(|\nabla \bu|) \nabla \bu\|_{W^{1,2}(\Omega, \rNn)} \leq C \big(\|{\rm {\bf div}} ( a(|\nabla \bu|) \nabla \bu)\|_{L^2(\Omega, \rN)}+ 
\| a(|\nabla \bu|) \nabla \bu\|_{L^1(\Omega, \rNn)}\big)
\,
\end{align}
for some constant $C=C(n,N, i_a, s_a, L_\Omega, d_\Omega, \mathcal  K)$, and for
 every function $\bu \in C^3(\Omega, \rN)\cap C^2(\overline \Omega, \rN)$ such that
\begin{equation}\label{dircond}
\bu =0 \qquad \hbox{on $\partial \Omega$.}
\end{equation}
%
%
%
In particular, if  $\Omega$ is convex, then inequality \eqref{fund} holds whatever $\mathcal K_\Omega$ is, and the constant $C$ in \eqref{fund} only depends on $n,N, i_a, s_a, L_\Omega, d_\Omega$. 
\end{lemma}

%
%
%
%
The following  gradient bound for solutions to the Dirichlet problem \eqref{eqdirbis} is needed to deal with  lower-order terms appearing in our global estimates. 

\begin{proposition}\label{talentil2}
 Assume that  $n\geq 2$, $N \geq 2$. Let $\Omega$ be an open set in $\rn$ such that $|\Omega|<\infty$.  Assume that the function $a: [0, \infty) \to [0, \infty)$ is continuously differentiable in $(0, \infty)$ and fulfills conditions \eqref{inf} and \eqref{sup}. Let  $\bff \in L^1(\Omega, \rN)\cap (W^{1,B}_0(\Omega, \rN))'$ and let $\bu$ be the weak solution to the Dirichlet problem \eqref{eqdirbis}. Then, there exists a constant $C=C(n, N, i_a, s_a,  |\Omega|)$ such that
\begin{equation}\label{talentigrad}
\| a(|\nabla \bu|) \nabla \bu\|_{L^{1}(\Omega , \rNn)} \leq C \|\bff\|_{L^1(\Omega, \rN)}.
\end{equation}
The same conclusion holds if  $\bff \in L^1(\Omega, \rN)$ and $\bu$ is an approximable solution to the Dirichlet problem \eqref{eqdirbis}.
\end{proposition}
\begin{proof}
%
Assume that  $\bff \in L^1(\Omega, \rN)\cap (W^{1,B}_0(\Omega, \rN))'$ and that $\bu$ is the weak solution to the Dirichlet problem \eqref{eqdirbis}.
Given $t>0$, let $T_t(\bfu): \Omega \to \rN$ be the function defined by
\begin{equation}\label{may80}
T_t(\bfu) = \begin{cases} \bfu & \quad \text{in $\{|\bfu|\leq t\}$}
\\ \displaystyle  t\frac{\bfu}{|\bfu|}  & \quad \text{in $\{|\bfu|> t\}$.}
\end{cases}
\end{equation}
Then $T_t(\bfu) \in W^{1,B}_0(\Omega, \rN)$, and 
\begin{equation}\label{may81}
\nabla T_t(\bfu) = \begin{cases} \nabla \bfu & \quad \text{a.e. in $\{|\bfu|\leq t\}$}
\\   \displaystyle \frac{t}{|\bfu|}  \Big(I - \frac{\bfu}{|\bfu|} \otimes \frac{\bfu}{|\bfu|} \Big) \nabla \bfu & \quad \text{a.e. in $\{|\bfu|> t\}$}
\end{cases}
\end{equation}
Observe that 
$$a(|P|)P \cdot (I-\omega \otimes \omega)P\geq 0$$
for every matrix $P \in \rNn$ and any vector $\omega \in \rN$ such that $|\omega|\leq 1$. Thus, on making use of  $T_t(\bfu)$ as a test function $\bfvarphi$ in equation \eqref{weaksol}, one deduces that
\begin{align}\label{may82}
\int_{\{|\bfu|\leq t\}} a(|\nabla \bfu|) |\nabla \bfu|^2\, dx & \leq \int_{\Omega} a(|\nabla \bfu|) \nabla \bfu \cdot \nabla T_t(\bfu)  \, dx = \int _\Omega \bff \cdot T_t(\bfu) \, dx 
\\ \nonumber & = \int _{\{|\bfu|\leq t\}} \bff \cdot \bfu \, dx + \int _{\{|\bfu>t\}} \bff \cdot  \displaystyle  t\frac{\bfu}{|\bfu|}\, dx \leq t \|\bff\|_{L^1(\Omega, \rN)}.
\end{align}
Hence, by the first inequality in \eqref{may101},
\begin{align}\label{may83}
\int_{\{|\bfu|\leq t\}} B(|\nabla \bfu|) \, dx\leq  t \|\bff\|_{L^1(\Omega, \rN)}.
\end{align}
On the other hand, the chain rule for vector-valued functions ensures that the function $|\bfu| \in W^{1,B}_0(\Omega)$, and $|\nabla \bfu| \geq |\nabla |\bfu||$ a.e. in $\Omega$. Inequality \eqref{may83} thus implies that
\begin{align}\label{may84} 
\int_{\{|\bfu|< t\}} B(|\nabla |\bfu||) \, dx\leq t \|\bff\|_{L^1(\Omega)} \quad \hbox{for $t>0$.}
\end{align}
The standard chain rule for Sobolev functions entails that
  $\T(|\bfu|) \in W^{1,B}(\Omega)$.  Let $\sigma > \max\{s_a +2, n\}$. Hence,  $\sigma > \max\{s_B, n\}$, inasmuch as $i_B\leq i_b+1 = i_a+2$. Owing to Lemma \ref{aux}, the assumptions of Theorem \ref{sobolev} are fulfilled,  with $A$ replaced by $B$ and with this choice of $\sigma$.  An application of the 
the Orlicz-Sobolev inequality \eqref{may95} to the function $\T (|\bfu|)$ tells us that
\begin{equation}\label{os1}
\int _\Omega B_\sigma\Bigg(\frac{|\T (|\bfu|)|}{C \big(\int _\Omega B(|\nabla \T (|\bfu|)|)dy\big)^{1/\sigma}}\Bigg)\, dx \leq \int _\Omega B(|\nabla (\T (|\bfu|))|)dx,
\end{equation}
where $C=c |\Omega|^{\frac 1n - \frac 1\sigma}$. Here, $B_\sigma$ denotes the function defined as in \eqref{Hsig}--\eqref{Bsig}, with $A$ replaced by $B$.
One has that
\begin{equation}\label{os4}
\int _\Omega B(|\nabla \T (|\bfu|)|)dx = \int _{\{|\bfu|< t\}} B(|\nabla |\bfu||)dx \quad \hbox{for $t>0$,}
\end{equation}
\begin{equation}\label{os5bis}
 |\T(|\bfu|)|=  t   \quad \hbox{in  $\{|\bfu|\geq t\}$,}
\end{equation}
and 
\begin{equation}\label{os5}
 \{|\T(|\bfu|)|\geq  t\} =  \{|\bfu|\geq  t\}  \quad \hbox{for $t>0$.}
\end{equation}
Thus, 
\begin{align}\label{os2}
|\{|\bfu|\geq t\}| B_\sigma\bigg(\frac{ t}{C (\int _{\{|\bfu|< t\}} B(|\nabla |\bfu||)dy)^{\frac 1\sigma} }\bigg )& \leq \int _{\{|\bfu|\geq t\}} B_\sigma\Bigg(\frac{|\T (|\bfu|)|}{C \big(\int _{\{|\bfu|< t\}} B(|\nabla |\bfu||)dy\big)^{1/\sigma}}\Bigg)\, dx \\ \nonumber & \leq  \int _{\{|\bfu|< t\}} B(|\nabla |\bfu||)dx 
\end{align}
for $t>0$.  Hence, by \eqref{may84},
\begin{align}\label{os3}
|\{|\bfu|\geq t\}| B_\sigma\bigg(\frac{ t}{C ( t \|\bff\|_{L^1(\Omega, \rN)})^{\frac 1\sigma}}\bigg)  \leq   t \|\bff\|_{L^1(\Omega, \rN)} \qquad \hbox{for $t > 0$.}
\end{align}
From inequality \eqref{may83} we deduce that
\begin{equation}\label{lem2.30}
|\{B(|\nabla \bfu|)>s, |\bfu|\leq t\}| \leq     \frac 1s \int_{\{B(|\nabla \bfu|)>s, |\bfu|\leq t\}} B(|\nabla\bfu|)\, dx \leq \frac{t \|\bff\|_{L^1(\Omega, \rN)}}s  \quad \hbox{for $t >0$ and $s > 0$.}
\end{equation}
Coupling inequalities \eqref{lem2.30} and \eqref{os3} yields
\begin{align}\label{lem2.4}
|\{B(|\nabla \bfu|)>s\}| & \leq  |\{|\bfu| >t\}| + |\{B(|\nabla \bfu|)>s,|\bfu|\leq t\}| \\ \nonumber & \leq 
\frac{t \|\bff\|_{L^1(\Omega, \rN)}}{B_\sigma(Ct^{\frac 1{\sigma'}}/(t \|\bff\|_{L^1(\Omega, \rN)})^{\frac 1\sigma})} + \frac{ t \|\bff\|_{L^1(\Omega, \rN)}}s \quad \hbox{for $t >0$ and $s>0$.}
\end{align}
The choice $t = \big(\tfrac 1C \|\bff\|_{L^1(\Omega, \rN)}^{1/\sigma} B_\sigma^{-1}(s)\big)^{\sigma '}$ in inequality  \eqref{lem2.4} results in
\begin{equation}\label{lem2.5}
|\{B(|\nabla  \bfu|)>s\}| \leq  \  \frac{2  \|\bff\|_{L^1(\Omega, \rN)}^{\sigma'}}{C^{\sigma'} }\frac{B_\sigma^{-1}(s)^{\sigma '}}{s} \quad \hbox{for $s>0$.}
\end{equation}
Next, 
set   $s= B(b^{-1}(\tau))$ in \eqref{lem2.5}  and make use of \eqref{Bsig} to obtain that
\begin{equation}\label{lem2.6}
|\{b(|\nabla  \bfu|)>\tau\}| \leq  \frac{2  \|\bff\|_{L^1(\Omega, \rN)}^{\sigma'}}{C^{\sigma'} } \frac{H_\sigma(b^{-1}(\tau ))^{\sigma '}}{B(b^{-1}(\tau ))}
%
%
\qquad \hbox{for $\tau >0$,}
\end{equation}
where $H_\sigma$ is defined as in \eqref{Hsig}, with $A$ replaced by $B$.
Thanks to inequality \eqref{lem2.6}, 
\begin{align}\label{may99}
\int_\Omega 
 b(|\nabla  \bfu|) \, dx & = \int _0^\infty |\{b(|\nabla  \bfu|)>\tau\}|\, d\tau 
\leq  \lambda b(|\Omega|) +    2  C^{-\sigma'}  \|\bff\|_{L^1(\Omega, \rN)}^{\sigma'}\int_\lambda ^\infty   \frac{H_\sigma(b^{-1}(\tau ))^{\sigma '}}{B(b^{-1}(\tau ))}\, d\tau
\end{align}
for $\lambda >0$.
 Owing to inequalities \eqref{may100} and \eqref{may101},  and to Fubinis's theorem, the following chain holds:
\begin{align}\label{may102}
\int_\lambda ^\infty &  \frac{H_\sigma(b^{-1}(\tau ))^{\sigma '}}{B(b^{-1}(\tau ))}\, d\tau  \leq (s_a+1)
\int_{b^{-1}(\lambda)} ^\infty   \frac{H_\sigma(s)^{\sigma '}}{sB(s)} b(s) ds \\\nonumber&= (s_a+1)
\int_{b^{-1}(\lambda)} ^\infty   \frac{b(s)}{sB(s)}  \int _0^s \bigg(\frac t{B(t)}\bigg)^{\frac 1{\sigma-1}}\,dt \,ds
\\ \nonumber & \leq (s_a+1)(s_a+2)
\int_{b^{-1}(\lambda)} ^\infty   \frac{1}{s^2}  \int _0^s \bigg(\frac t{B(t)}\bigg)^{\frac 1{\sigma-1}}\,dt \,ds
\\ \nonumber &= (s_a+1)(s_a+2)\bigg(\int_0^{b^{-1}(\lambda)}  \bigg(\frac t{B(t)}\bigg)^{\frac 1{\sigma-1}} \int_{b^{-1}(\lambda)} ^\infty   \frac{ds}{s^2} \, dt + \int _{b^{-1}(\lambda)} ^\infty \bigg(\frac t{B(t)}\bigg)^{\frac 1{\sigma-1}} \int_{t} ^\infty   \frac{ds}{s^2} \, dt\bigg)
\\ \nonumber &= (s_a+1)(s_a+2)\bigg(\frac{1}{b^{-1}(\lambda)}   \int_0^{b^{-1}(\lambda)}  \bigg(\frac t{B(t)}\bigg)^{\frac 1{\sigma-1}} \, dt + \int _{b^{-1}(\lambda)} ^\infty \bigg(\frac t{B(t)}\bigg)^{\frac 1{\sigma-1}}   \frac{dt}{t}\bigg)
\\ \nonumber &\leq  (s_a+1)(s_a+2)^{\sigma'}\bigg(\frac{1}{b^{-1}(\lambda)}   \int_0^{b^{-1}(\lambda)}  \bigg(\frac 1{b(t)}\bigg)^{\frac 1{\sigma-1}} \, dt + \int _{b^{-1}(\lambda)} ^\infty \bigg(\frac 1{b(t)}\bigg)^{\frac 1{\sigma-1}}   \frac{dt}{t}\bigg)
\end{align}
for $\lambda >0$.
The function 
$\frac {t^{s_a+1+\ep}}{b(t)}$
is increasing for every $\ep>0$. Hence, if  $0<\ep < \sigma - s_a -2$, then   
\begin{align}\label{may103}
\frac{1}{b^{-1}(\lambda)} &  \int_0^{b^{-1}(\lambda)}  \bigg(\frac 1{b(t)}\bigg)^{\frac 1{\sigma-1}} \, dt  =
\frac{1}{b^{-1}(\lambda)}   \int_0^{b^{-1}(\lambda)}  \bigg(\frac {t^{s_a+1+\ep}}{b(t)}\bigg)^{\frac 1{\sigma-1}} t^{- \frac{s_a+1+\ep}{\sigma -1}}\, dt \\ \nonumber & \leq \frac{1}{b^{-1}(\lambda)}  \bigg(\frac {b^{-1}(\lambda)^{s_a+1+\ep}}{\lambda}\bigg)^{\frac 1{\sigma-1}}   \int_0^{b^{-1}(\lambda)}  t^{- \frac{s_a+1+\ep}{\sigma -1}}\, dt = \frac{\sigma -1}{\sigma - s_a-2-\ep} \lambda ^{-\frac 1{\sigma -1}} \quad \text{for $\lambda > 0$.}
\end{align}
On the other hand, if $0<\ep < i_a +1$, then the  function 
$\frac {t^{\ep}}{b(t)}$
is decreasing. Hence, 
\begin{align}\label{may104}
 \int _{b^{-1}(\lambda)} ^\infty \bigg(\frac 1{b(t)}\bigg)^{\frac 1{\sigma-1}}   \frac{dt}{t}  
 = \int _{b^{-1}(\lambda)} ^\infty \bigg(\frac {t^\ep}{b(t)}\bigg)^{\frac 1{\sigma-1}}  t^{-\frac \ep{\sigma -1}-1}\, dt &\leq  \bigg(\frac {b^{-1}(\lambda)^\ep}{\lambda}\bigg)^{\frac 1{\sigma-1}}  \int _{b^{-1}(\lambda)} ^\infty t^{-\frac \ep{\sigma -1}-1}\, dt
\\ \nonumber &= \frac{\sigma -1}\ep \lambda ^{-\frac 1{\sigma -1}} \quad \text{for $\lambda > 0$.}
\end{align}
Inequalities \eqref{may102}--\eqref{may104} entail that there exists a constant $c=c(\sigma, i_a, s_a)$ such that
\begin{align}\label{may105}
\int_\lambda ^\infty   \frac{H_\sigma(b^{-1}(\tau ))^{\sigma '}}{B(b^{-1}(\tau ))}\, d\tau \leq c \lambda ^{-\frac 1{\sigma -1}} \quad \text{for $\lambda >0$.}
\end{align}
Inequality \eqref{talentigrad} follows from \eqref{may99} and \eqref{may105}, with the choice $\lambda = \|\bff\|_{L^1(\Omega, \rN)}$. 
\\ The assertion about the case when $\bff \in L^1(\Omega, \rN)$ and $\bu$ is an approximable solution to the Dirichlet problem \eqref{eqdirbis} follows on applying inequality \eqref{talentigrad} with $\bff$ and $\bu$ replaced by the functions $\bff_k$ and $\bu_k$ appearing in the definition of approximable solutions, and passing to the limit as $k \to \infty$ in the resultant inequality.  Fatou's lemma plays a role here.
\end{proof}

A last preliminary result,  proved in \cite[Lemma 5.2]{CiMa_JMPA}, is needed in an approximation argument for the domain $\Omega$ in our proof of Theorem \ref{seconddir}.

\begin{lemma}\label{approxcap}
Let $\Omega$ be a bounded Lipschitz domain in $\rn$, $n \geq 2$ such that $\partial \Omega \in W^{2,1}$. Assume that the function $\mathcal K_\Omega (r)$, defined as in \eqref{defK}, is finite-valued for $r\in (0,1)$.
Then there exist positive constants $r_0$ and $C$ and a sequence of bounded open sets $\{\Omega_m\}$, 
such that $\partial \Omega _m \in C^\infty$, $\Omega \subset \Omega _m$, $\lim _{m \to \infty}|\Omega _m \setminus \Omega| = 0$, the Hausdorff distance between $\Omega _m$ and $\Omega$ tends to $0$ as $m \to \infty$,
\begin{equation}\label{feb35}
L_{\Omega _m} \leq C L_\Omega \,, \quad d_{\Omega _m} \leq C d_\Omega
\end{equation}
and
\begin{equation}\label{appcap0}
\mathcal K_{\Omega_m}(r) \leq C \mathcal K_{\Omega} (r)
\end{equation}
for $r\in (0, r_0)$ and $m \in \mathbb N$.
\end{lemma}

\smallskip
\par\noindent
\begin{proof}[Proof of Theorem \ref{seconddir}]
It suffices to prove Part (i). Part (ii) will then follow,  since, by \cite[Lemmas 3.5 and 3.7]{CiMa_JMPA}, 
\begin {equation}\label{gen3} \mathcal K_{\Omega} (r) \leq C \sup_{x\in \partial {\Omega}}\|\mathcal B\|_{X(\partial {\Omega} \cap B_r(x))} \qquad \hbox{for $r \in (0, r_0)$,}
\end{equation}
for suitable constants $r_0$ and $C$ depending on $n$, $L_{\Omega}$ and $d_{\Omega}$.
\\ We split the proof  in three separate steps, where approximation arguments for the differential operator, the domain and the datum on the right-hand side of the system, respectively, are provided.
\\ \emph{Step 1}. Assume that
the  additional conditions 
 \begin{equation}\label{fsmooth}
\bff \in C^\infty_0(\Omega, \rN)\,,
\end{equation}
 and 
 \begin{equation}\label{omegasmooth}
\partial \Omega \in C^\infty\,
\end{equation}
are in force.
Given $\ep \in (0,1)$,  we denote by $\bu_\varepsilon$  the weak   solution to the system
\begin{equation}\label{eqdirichletep}
\begin{cases}
- {\rm {\bf  div}}(a_\varepsilon  (|\nabla \bu_\varepsilon|) \nabla {\bfu}_\varepsilon ) = {\bff} & {\rm in}\,\,\, \Omega \\
 {\bfu_\varepsilon} =0  &
{\rm on}\,\,\,
\partial \Omega \,,
\end{cases}
\end{equation}
where  $a_\ep$ is the function defined by \eqref{aeps}. An application of \cite[Theorem 2.1]{cmARMA} tells us that
\begin{align}\label{gradbound}
\|\nabla \bu_\varepsilon\|_{L^\infty (\Omega, \rNn)} \leq C
\end{align}
for some constant $C$ independent of $\ep$. Let us notice that the statement of \cite[Theorem 2.1]{cmARMA} yields inequality \eqref{gradbound} under   the assumption that the function $a_\epsilon$ be either increasing or decreasing; such an additional assumption can however be dropped via a slight variant in the proof.
Inequality \eqref{gradbound} implies that,
 for each $\varepsilon \in (0,1)$, 
\color{black}
\begin{equation}\label{abound}
 c_1 \leq a_\varepsilon (|\nabla \bu_\varepsilon|) \leq c_2 \qquad \hbox{in $\Omega$}
\end{equation}
for suitable positive constants $c_1$ and $c_2$.
\\  A classical  result by Elcrat and Meyers \cite[Theorem 8.2]{BF} enables us to deduce, via properties  \eqref{fsmooth}, \eqref{omegasmooth} and \eqref{abound}, that
${\bf u_\varepsilon}\in W^{2,2}(\Omega, \rN )$. 
Consequently, ${\bf u}_\varepsilon \in W^{1,2}_0(\Omega, \rn) \cap W^{1,\infty}(\Omega, \rN) \cap
W^{2,2}(\Omega, \rN)$. One can then find a
sequence $\{{\bf u}_k\} \subset C^\infty (\Omega, \rN)\cap
C^2(\overline \Omega, \rN)$ such that $\bu_k = 0$ on $\partial
\Omega$ for $k \in \mathbb N$, and
\begin{equation}\label{convk}
{\bfu}_k \to {\bfu}_\varepsilon \quad \hbox{in $W^{1,2}_0(\Omega, \rN)$,}
\quad {\bfu}_k \to {\bfu}_\varepsilon \quad \hbox{in $W^{2,2}(\Omega, \rN)$,}
\quad \nabla {\bf u}_k \to \nabla {\bfu}_\varepsilon 
\quad \hbox{a.e. in
$\Omega $},
\end{equation}
as $k \to \infty$ -- see e.g.   \cite[Chapter 2, Corollary 3]{Burenkov}. One also has that
\begin{equation}\label{boundk}
\|\nabla \bu_k\|_{L^\infty (\Omega, \rNn)}  \leq C \|\nabla \bu_\varepsilon\|_{L^\infty (\Omega, \rNn)} 
\end{equation}
\color{black}
for some constant $C$ independent of $k$, 
and,
%
%
 by the chain rule for
vector-valued Sobolev functions \cite[Theorem 2.1]{MarcusMizel}, $|\nabla |\nabla {\bfu}_k|| \leq |\nabla ^2
\bu_k|$ a.e. in $\Omega$. Moreover,    \cite[Equation (6.12)]{cmARMA} tells us that
\begin{align}\label{sys1k}
- {\rm {\bf div}} (a_\varepsilon(|\nabla {\bu_k}|)\nabla {\bfu_k} ) \to {\bff}
\quad \hbox{in $L^2(\Omega , \rN)$},
\end{align}
as $k \to \infty$. Assumption \eqref{capcond} enables us to apply  inequality \eqref{fund},  with $a$ replaced by $a_\varepsilon$ and $\bu$ replaced by $\bu _k$, to deduce that
\begin{align}\label{main16k}
\| a_\varepsilon(|\nabla \bu_k|)\nabla \bu_k\|_{W^{1,2}(\Omega, \rNn)}  \leq C \Big(
\|{\rm  {\bf div}} ( a_\varepsilon(|\nabla \bu_k|) \nabla \bu_k)\|_{L^2(\Omega, \rN)} +
 \| a_\varepsilon(|\nabla \bu_k|)\nabla \bu_k\|_{L^1(\Omega, \rNn)}\Big)
\end{align}
for $k\in \mathbb N$, and for some constant $C=C(n, N, i_a, s_a, L_\Omega, d_\Omega, \mathcal K_\Omega)$.  Notice that  this constant actually depends on the function $a_\ep$ only through $i_a$ and $s_a$, and it is hence independent of $\ep$, owing to \eqref{epsindex}.
Equations  \eqref{boundk}--\eqref{main16k} ensure that the sequence $\{a_\varepsilon(|\nabla \bu_k|)\nabla \bu_k\}$ is bounded in $W^{1,2}(\Omega, \rNn)$. Therefore, there exist a subsequence of $\{\bu_k\}$, still denoted by $\{\bu_k\}$, and a function ${\bf U}_\varepsilon \in W^{1,2}(\Omega, \rNn)$ such that
\begin{equation}\label{main32keps}
a_\varepsilon(|\nabla \bu_{k}|) \nabla \bu_{k} \to {\bf U}_\varepsilon\quad \hbox{in $L^2(\Omega, \rNn )$,} \quad a_\varepsilon(|\nabla \bu_{k}|) \nabla \bu_{k} \rightharpoonup {\bf U}_\varepsilon\quad \hbox{in $W^{1,2}(\Omega, \rNn )$}.
\end{equation}
Equation \eqref{convk} entails that $\nabla \bu_k \to \nabla \bu_\varepsilon$ a.e. in $\Omega$. As a consequence,
\begin{equation}\label{main31keps}
a_\varepsilon(|\nabla \bu_{k}|)\nabla \bu_{k} \to a_\varepsilon(|\nabla \bu_\varepsilon|)\nabla \bu_\varepsilon \quad \hbox{a.e. in $\Omega$.}
\end{equation}
From equations \eqref{main31keps} and \eqref{main32keps} one infers that
\begin{equation}\label{main100eps}
a_\varepsilon(|\nabla \bu_\varepsilon|) \nabla \bu_\varepsilon  = {\bf U}_\varepsilon \in W^{1,2}(\Omega, \rNn )\,.
\end{equation}
Furthermore,  passing to the limit as $k \to \infty$ in \eqref{main16k}
yields
\begin{align}\label{main16eps}
\| a_\varepsilon(|\nabla \bu_\varepsilon|)\nabla \bu_\varepsilon\|_{W^{1,2}(\Omega, \rNn)} \leq C \big(
\|{\bff}\|_{L^2(\Omega, \rN) } + 
 \|a_\varepsilon(|\nabla \bu_\varepsilon|)\nabla \bu_\varepsilon\|_{L^1(\Omega, \rNn)}\big).
\end{align}
Here,
equations  \eqref{main32keps} and \eqref{main100eps} have been exploited to pass to the limit on the left-hand side, and equations   \eqref{boundk} and \eqref{sys1k} on the right-hand side.  Combining  equations \eqref{main16eps} and \eqref{gradbound} tells us that 
\begin{align}\label{main16eps'}
\| a_\varepsilon(|\nabla \bu_\varepsilon|)\nabla \bu_\varepsilon\|_{W^{1,2}(\Omega, \rNn)} \leq C
\end{align}
for some constant   $C$, independent of $\ep$.
By the last inequality, the family of functions  $\{a_\varepsilon(|\nabla \bu_\varepsilon|)\nabla \bu_\varepsilon\}$ is uniformly bounded in $W^{1,2}(\Omega, \rNn)$ for $ \ep \in (0,1)$.  Therefore, there exist a sequence $\{\varepsilon _m\}$ converging to $0$ and a function ${\bf U}  \in W^{1,2}(\Omega, \rNn)$ such that  
\begin{equation}\label{main32eps}
a_{\varepsilon _m}(|\nabla \bu_{\varepsilon _m}|) \nabla \bu_{\varepsilon _m} \to {\bf U} \,\,\hbox{in $L^2(\Omega, \rNn)$,} \quad a_{\varepsilon _m}(|\nabla \bu_{\varepsilon _m}|) \nabla \bu_{\varepsilon _m} \rightharpoonup {\bf U} \, \,\hbox{in $W^{1,2}(\Omega, \rNn )$}.
\end{equation}
An  argument parallel to that of  the proof of \eqref{may48} yields
\begin{align}\label{akmglob}
\nabla \bu _{\ep_m} \to \nabla \bu \qquad \hbox{a.e. in $\Omega$.}
\end{align}
We omit the details, for brevity. Let us just point out that, 
 in this argument, one has to make use of the inequality 
\begin{align}\label{cialpha1glob}
\int _{\Omega} B(|\nabla \bu _{\ep_m}|)\, dx \leq C \bigg(\int _{\Omega}\widetilde B(|\bff|)\,dx + B(\ep_m)\bigg)\,,
\end{align}
instead of \eqref{cialpha1}. Inequality \eqref{cialpha1glob} easily follows on choosing
  $\bu _{\ep_m}$ as a test function in the definition of weak solution to problem \eqref{eqdirichletep}, with $\ep = \ep_m$.
Coupling equations \eqref{main32eps} and \eqref{akmglob} implies that
\begin{equation}\label{main25}
a(|\nabla \bu|) \nabla \bu = {\bf U} \in W^{1,2}(\Omega, \rNn)\,.
\end{equation}
On the other hand, on exploiting equations  \eqref{akmglob} and \eqref{gradbound},  the dominated convergence theorem for Lebesgue integrals and 
inequality \eqref{talentigrad} (applied with $a$ and $\bfu$ replaced by $a_{\varepsilon _m}$ and $\bu_{\varepsilon_m}$) one can deduce that
\begin{align}\label{napoli}
\lim _{m \to \infty}  \|a_{\varepsilon _m}(|\nabla \bu_{\varepsilon_m}|)\nabla \bu_{\varepsilon_m}\|_{L^1(\Omega, \rNn)} = 
\|a(|\nabla \bu|)\nabla u\|_{L^{1}(\Omega, \rNn)} \leq C \|\bff\|_{L^2(\Omega, \rN)}
\end{align}
for some constant $C=C(n,N,i_a, s_a, |\Omega|)$. 
Combining equations \eqref{main16eps},  \eqref{main32eps}, \eqref{main25} and \eqref{napoli} yields
\begin{equation}\label{main27}
\|a(|\nabla \bu|) \nabla \bu\|_{W^{1,2}(\Omega, \rNn)} \leq C \|{\bff}\|_{L^2(\Omega, \rN)}
\end{equation}
%
%
for some constant $C= C(n, N, i_a, s_a, L_\Omega, d_\Omega, \mathcal K_\Omega)$.

\smallskip
\par\noindent
\emph{Step 2}. Assume now that the temporary condition \eqref{fsmooth} is still in force, but $\Omega$ is just as in the statement.  Let $\{\Omega_m\}$ be a sequence of open sets approximating $\Omega$  in the sense of Lemma \ref{approxcap}. For each $m \in \mathbb N$, denote by $\bfu_m$  the weak solution to the Dirichlet problem 
\begin{equation}\label{eqm}
\begin{cases}
- {\rm {\bf div}} (a(|\nabla \bu_m|)\nabla \bu_m ) = {\bff}  & {\rm in}\,\,\, \Omega _m \\
 \bu_m =0  &
{\rm on}\,\,\,
\partial \Omega _m \,,
\end{cases}
\end{equation}
where ${\bff}$  is continued by $0$ outside $\Omega$. Owing to our assumptions on $\Omega_m$, inequality \eqref{main27} holds for $\bfu_m$. Thereby, there exists a constant  $C(n, N, i_a, s_a, L_\Omega, d_\Omega, \mathcal K_\Omega)$ such that
\begin{align}\label{main27aus}
\|a(|\nabla \bu_m|) \nabla \bu_m\|_{W^{1,2}(\Omega, \rNn)} \leq  
\|a(|\nabla \bu_m|) \nabla \bu_m\|_{W^{1,2}(\Omega_m, \rNn)}    \leq C \|{\bff}\|_{L^2(\Omega_m, \rN)}= C \|{\bff}\|_{L^2(\Omega, \rN)}.
\end{align}
%
%
Observe that the dependence of the constant $C$ on the specified quantities, and, in particular, its independence of $m$, is due to
properties  \eqref{feb35} and \eqref{appcap0}.
\\
Thanks to \eqref{main27aus}, the sequence  
$\{a(|\nabla \bu_m |) \nabla  \bu_m\}$ is  bounded in $W^{1,2}(\Omega, \rNn)$,
and hence  there exists a subsequence, still denoted by $\{\bu_m\}$ and a function  ${\bfU} \in W^{1,2}(\Omega, \rNn)$, such that
\begin{align}\label{main32}
a(|\nabla \bu_{m}|) \nabla \bu_{m} &\to {\bfU}\quad \hbox{in $L^2(\Omega, \rNn)$,} \\\nonumber   a(|\nabla \bu_{m}|)\nabla \bu_{m} &\rightharpoonup {\bfU}\quad \hbox{in $W^{1,2}(\Omega, \rNn)$}.
\end{align}
We now notice that 
there exists $\alpha \in (0,1)$, independent of $m$, such that $\bu _m\in C^{1,\alpha}_{\rm loc}(\Omega, \rN)$, and for every  open set $\Omega ' \subset \subset \Omega$ with a smooth boundary
\begin{align}\label{main29}
\|\bu_{m}\|_{C^{1,\alpha}(\Omega ', \rN)} \leq C,
\end{align}
for some   $C$, independent of $m$.
To verify this assertion, one can make use of  \cite[Corollary 5.5]{DKS}
%
and of inequality \eqref{holder1} to deduce that, for each open set $\Omega'$ as above, there exists a constant $C$, independent of $m$, such that
\begin{align}\label{calpham}
\|\nabla \bu _{m}\|_{C^{\alpha}(\Omega ', \rNn)} \leq C.
\end{align}
Since the function $\bff$ satisfies assumption \eqref{fsmooth}, a basic energy estimate for weak solutions tells us that
\begin{align}\label{energyst}
\int_{\Omega_m} B(|\nabla \bu _{m}|)\, dx \leq C
\end{align} 
for some constant $C$ independent of $m$. Thus, as a consequence of  the Poincar\'e  inequality \eqref{poinc}, 
\begin{align}\label{energym}
 \int_{\Omega_m} B(|\bu _{m}|)\, dx \leq C\,,
\end{align}
 where the constant  $C$ is independent of $m$, for $|\Omega _m|$ is uniformly bounded. 
%
%
Inequalities  
\eqref{calpham}
 and \eqref{energym},   via a Sobolev type inequality,  tell us that
\begin{align}\label{linfinitym}
 \| \bu _{m}\|_{L^\infty(\Omega ', \rN)} \leq C\,
\end{align}
for some constant $C$ independent of $m$.
Inequality \eqref{main29} follows from \eqref{calpham} and \eqref{linfinitym}.
\color{black}
\\  On passing, if necessary,  to another subsequence,  we deduce from inequality \eqref{main29} that there exists a function ${\bfv} \in C^1(\Omega, \rN)$ such that
\begin{equation}\label{main30}
\bu_{m} \to {\bfv}\, \quad \hbox{and}\quad \nabla \bu_{m} \to \nabla {\bfv} \quad \hbox{in $\Omega$.}
\end{equation}
Hence,
\begin{equation}\label{main31}
a(|\nabla \bu_{m}|) \nabla \bu_{m} \to a(|\nabla {\bf v}|) \nabla {\bfv} \quad \hbox{in $\Omega$.}
\end{equation}
Owing to equations \eqref{main31} and \eqref{main32}, 
\begin{equation}\label{main100}
a(|\nabla {\bfv}|) \nabla {\bfv}   = {\bfU} \in W^{1,2}(\Omega, \rNn)\,.
\end{equation}
 Next,  we pick  a test  function $\bfvarphi \in C^\infty _0(\Omega, \rN)$ (continued by $0$ outside $\Omega$) in the  definition of weak solution to problem \eqref{eqm}.  
  Passing to the limit as $m \to \infty$ in the resulting equation 
yields, via \eqref{main32} and \eqref{main100},
\begin{equation}\label{main101bis}
\int _\Omega a(|\nabla {\bfv}|) \nabla {\bfv} \cdot \nabla \bfvarphi \, dx = \int_\Omega {\bff} \cdot \bfvarphi \, dx\,.
\end{equation}
Inequality \eqref{energyst} tells us that $\int_\Omega B(|\nabla \bu _m|) \, dx  \leq C$ for some constant $C$ independent of $m$. Therefore, this  inequality is still true if $\bu _m$ is replaced with ${\bfv}$. 
Consequently, thanks to inequality \eqref{may41}, $\int_\Omega \widetilde B(a(|\nabla {\bf v}|) |\nabla {\bfv}|)\, dx  <\infty$. Thus, since under our assumptions on $a$ the space $C^\infty _0(\Omega, \rN)$ is dense in $W^{1,B}_0(\Omega, \rN)$,   equation \eqref{main101bis}   holds for every function $\bfvarphi \in W^{1,B}_0(\Omega, \rN)$ as well. Hence,  ${\bfv}$ is a weak solution to the Dirichlet problem \eqref{eqdir}, and, inasmuch as such a solution is unique,    ${\bfv}=\bu$.
Moreover, by equations \eqref{main27aus} and \eqref{main32},  there exists a   constant $C=C(n, N, i_a, s_a, L_\Omega, d_\Omega, \mathcal K_\Omega)$ such that
\begin{equation}\label{main35}
\|a(|\nabla \bu|) \nabla \bu\|_{W^{1,2}(\Omega, \rNn)} \leq C \|{\bff}\|_{L^2(\Omega, \rN)}.
\end{equation}
%
%

\smallskip
\par\noindent
\emph{Step 3}. Finally, assume that both $\Omega$ and $\bff$ are as in  the statement.
The definition of approximable solution entails that there exists a  sequence
$\{{\bff}_k\} \subset C^\infty_0(\Omega, \rN )$, such that  ${\bff}_k \to {\bff}$ in $L^2(\Omega, \rN)$ and   the sequence of weak solutions 
 $\{\bu _k\} \subset W^{1,B}_0(\Omega, \rN)$ to problems \eqref{eqdirichletk}, fulfills
  $\bu _k \to \bu$ and $\nabla \bu_k \to \nabla \bu$ a.e. in $\Omega$. An application of inequality \eqref{main35}   with $\bfu$ and $\bff$ replaced by $\bfu _k$ and $\bff_k$, tells us that   $a(|\nabla \bu_k|)\nabla \bu_k \in W^{1,2}(\Omega, \rNn)$, and 
\begin{align}\label{main40}
\|a(|\nabla \bu_k|) \nabla \bu_k\|_{W^{1,2}(\Omega, \rNn)} \leq C_1 \|{\bff}_k\|_{L^2(\Omega, \rN)} \leq C_2 \|{\bff}\|_{L^2(\Omega, \rN)}\,
\end{align}
for some  constants $C_1$ and $C_2$, depending on  $N$, $i_a$, $s_a$ and $\Omega$. Therefore, the sequence $\{a(|\nabla \bu_k|)\nabla \bu_k\}$ is  bounded in $ W^{1,2}(\Omega, \rNn)$, whence there exists a subsequence, still indexed by $k$, and a function ${\bfU} \in W^{1,2}(\Omega, \rNn)$ such that
\begin{equation}\label{main38}
a(|\nabla \bu_{k}|)\nabla \bu_{k} \to {\bfU}\quad \hbox{in $L^2(\Omega, \rNn)$,} \quad  a(|\nabla \bu_{k}|) \nabla \bu_{k} \rightharpoonup {\bf U}\quad \hbox{in $W^{1,2}(\Omega, \rNn)$}.
\end{equation}
Inasmuch as  $\nabla \bu_k \to \nabla \bu$ a.e. in $\Omega$, one hence deduces that $a(|\nabla \bu|) \nabla \bu ={\bfU}  \in W^{1,2}(\Omega, \rNn)$. 
Thereby, the second inequality in \eqref{secondconvex1} follows from equations \eqref{main40} and     \eqref{main38}.  The first inequality  in \eqref{secondconvex1} holds trivially. 
The proof is complete.
\end{proof}

\medskip
\par\noindent
\begin{proof}[Proof of Theorem \ref{secondconvexteo}]  The proof parallels that of Theorem \ref{seconddir}.  However, \emph{Step 2} requires a variant. The sequence $\{\Omega _m\}$ of bounded
sets, with smooth boundaries, coming into play in this step has to be chosen in such a way that they are convex and 
 approximate  ${\Omega}$ from outside with respect to the Hausdorff distance. Inequalities   \eqref{feb35}   automatically hold in this case. Moreover, inequality  \eqref{appcap0} is not needed, inasmuch as   the constant $C$ in  \eqref{fund} does not depend on the function $\mathcal K_{\Omega}$ if  ${\Omega}$ is convex.
\end{proof}

\section*{Compliance with Ethical Standards}\label{conflicts}

\smallskip
\par\noindent 
{\bf Funding}. This research was partly funded by:  
\\(i) German Research
  Foundation (DFG)  through the CRC 1283 in Bielefeld University (A. Kh.Balci and L. Diening);
\\ (ii) Research Project   of the Italian Ministry of Education, University and
Research (MIUR) Prin 2017 ``Direct and inverse problems for partial differential equations: theoretical aspects and applications'',
grant number 201758MTR2 (A. Cianchi);
\\ (iii) GNAMPA   of the Italian INdAM - National Institute of High Mathematics (grant number not available)  (A. Cianchi);   
\\  (iv) RUDN University Strategic Academic Leadership Program (V. Maz'ya).
\smallskip
\par\noindent
{\bf Conflict of Interest}. The authors declare that they have no conflict of interest.

\printbibliography 

\end{document}